%% file: gradflow.tex
\documentclass[11pt,a4paper]{article}

\usepackage{a4wide}
\usepackage{graphicx}
\usepackage{latexsym}
\usepackage{epsfig}
\usepackage{amssymb}
\usepackage{amstext}
\usepackage{amsgen}
\usepackage{amsxtra}
\usepackage{amsgen}
\usepackage{amsthm}
\usepackage{enumerate,amsmath,subfigure,color}
\usepackage{tikz}

\newtheorem{thm}{Theorem}[section]
\newtheorem{prop}[thm]{Proposition}
\newtheorem{lemma}[thm]{Lemma}
\newtheorem{cor}[thm]{Corollary}
\newtheorem{rem}[thm]{Remark}

\theoremstyle{definition}
\newtheorem{example}[thm]{Example}
\newtheorem{definition}[thm]{Definition}

\numberwithin{equation}{section}

\newcommand{\N}{\mathbb{N}}

\newcommand{\R}{\mathbb{R}}

\newcommand{\dom}{{\mathrm{dom}}}
\newcommand{\norm}[1]{\left\Vert #1 \right\Vert}

\renewcommand{\H}{{\cal H}}
\newcommand{\X}{{\cal X}}

\renewcommand{\div}{\mathrm{div}\,}
\newcommand{\rot}{\mathrm{rot}\,}

\newcommand{\argmin}{\mathrm{arg}\min}
\newcommand{\st}{\;:\;}
\newcommand{\testSpaceV}{C^\infty_c(\Omega,[-1,1])}
\newcommand{\spaceV}{H^1_0(\Omega,[-1,1])}
\newcommand{\E}{\mathcal{E}}
\newcommand{\supp}{\mathrm{supp}}
\newcommand{\dx}{\,\mathrm{d}x}
\renewcommand{\d}{\,\mathrm{d}}

\newcommand{\calN}{\mathcal{N}}
\newcommand{\calJ}{\mathcal{J}}
\newcommand{\calC}{\mathcal{C}}
\newcommand{\calK}{\mathcal{K}}
\newcommand{\sgn}{\mathrm{sgn}}
\newcommand{\conv}{\mathrm{conv}}
\newcommand{\eps}{\varepsilon}
\newcommand{\id}{\mathrm{id}}
\newcommand{\bv}{{BV}}
\newcommand{\tv}{\mathrm{TV}}
\newcommand{\change}[1]{\textcolor{red}{#1}}
\renewcommand{\change}[1]{#1}

\newcommand{\secchange}[1]{\textcolor{red}{#1}}
\renewcommand{\secchange}[1]{#1}

\begin{document}

\title{Nonlinear Spectral Decompositions by Gradient Flows of One-Homogeneous Functionals}
\author{Leon Bungert\thanks{Department Mathematik, Universit\"{a}t Erlangen-N\"{u}rnberg, Cauerstrasse 11, 91058 Erlangen, Germany. \texttt{\{leon.bungert,martin.burger\}@fau.de}} \and Martin Burger\footnotemark[1] \and Antonin Chambolle\thanks{CMAP, \'{E}cole Polytechnique, CNRS, 91128 Palaiseau Cedex,
France. \texttt{antonin.chambolle@cmap.polytechnique.fr}} \and Matteo Novaga\thanks{Dipartimento di Matematica, Universit\`{a} di Pisa, Largo Pontecorvo 5, 56127 Pisa, Italy. \texttt{matteo.novaga@unipi.it}}}
\maketitle

\begin{abstract}
This paper establishes a theory of nonlinear spectral decompositions by considering the eigenvalue problem related to an absolutely one-homogeneous functional in an infinite-dimensional Hilbert space. This approach is both motivated by works for the total variation, where interesting results on the eigenvalue problem and the relation to the total variation flow have been proven previously, and by recent results on finite-dimensional polyhedral semi-norms, where gradient flows can yield spectral decompositions into eigenvectors. 

We provide a geometric characterization of eigenvectors via a dual unit ball and prove them to be subgradients of minimal norm. This establishes the connection to gradient flows, whose time evolution is a decomposition of the initial condition into subgradients of minimal norm. If these are eigenvectors, this implies an interesting orthogonality relation and the equivalence of the gradient flow to a variational regularization method and an inverse scale space flow. Indeed we verify that all scenarios where these equivalences were known before by other arguments -- such as one-dimensional total variation, multidimensional generalizations to vector fields, or certain polyhedral semi-norms -- yield spectral decompositions, and we provide further examples. We also investigate extinction times and extinction profiles, which we characterize as eigenvectors in a very general setting, generalizing several results from literature.


\end{abstract}

\section{Introduction}

%\subsection{Motivation}
Spectral properties and spectral decompositions are at the heart of many arguments in mathematics and physics, let us just mention the spectral decomposition of self-adjoint linear operators (cf. \cite{reed1972methods,kato2013perturbation}) or the eigenvalue problems for high-dimensional or nonlinear Schr\"odinger equations (cf. \cite{weinstein1985modulational}) as two prominent examples. In signal and image processing a variety of successful approaches were based on Fourier transforms and Laplacian eigenfunctions, in image reconstruction and inverse problems the singular value decomposition is the fundamental tool for linear problems. Over the last two decades variational approaches and other techniques such as sparsity in anisotropic Banach spaces became popular and spectral techniques lost their dominant roles (cf. \cite{benning2018modern,burger2013guide,caselles2014total,donoho2006compressed,mairal2014sparse}). 

Standard examples considered in the nonlinear setting are gradient flows of the form
$$ \partial_t u(t)=-p(t),\;p(t)\in\partial J (u(t)),\quad u(0)=f,$$
in some Hilbert space with $\partial J$ denoting the subdifferential of a semi-norm (without Hilbertian structure in general) on a dense subspace, and variational problems of the form
$$ \frac{1}2 \Vert u  -f \Vert^2 + t J(u) \rightarrow \min_u $$
with the norm in the first term being the one in the Hilbert space.  
As some recent results demonstrate, the role of eigenvalue problems and even spectral decompositions may be underestimated in such settings. First of all, data $f$ satisfying the nonlinear eigenvector relation
$$\lambda f \in \partial J(f)$$
for a scalar $\lambda$ give rise to analytically computable solutions for such problems, which was made precise for the TV flow (cf. \cite{andreu2002some,bellettini2002total})  and is hidden in early examples for the variational problem (cf. \cite{meyer2001oscillating} and \cite{benning2013ground} for a detailed discussion). Secondly, for a general datum $f$ the solution of the gradient flow satisfies
$$f-\overline{f}=\int_0^\infty p(t)\d t,$$
i.e., the datum is decomposed into subgradients $p(t)$ of the functional $J$. \change{Here $\overline{f}$ denotes component of the datum $f$ which is in the null-space of $J$ and is left invariant under the flow, for instance the mean value of a function in the case $J=\tv$.} In the case, that these subgradients are even eigenvectors, this is called a \emph{nonlinear spectral decomposition}.
 
In \cite{benning2013ground} some further interesting properties of nonlinear eigenvectors (in a more general setting) such as their use for scale estimates and several relevant examples have been provided. A rather surprising conjecture was made by Gilboa (cf. \cite{gilboa2014total}), suggesting that TV flow and similar schemes can provide a spectral decomposition, i.e., time derivatives of the solution are related to eigenfunctions of the total variation. This was made precise in \cite{burger2016spectral} in a certain finite-dimensional setting; furthermore, scenarios where a decomposition into eigenvectors can be computed were investigated. In more detail, functionals of the form $J(u)=\Vert A u \Vert_1$ with a matrix $A$ such that $AA^T$ is diagonally dominant lead to such spectral decompositions. 

%\subsection{Plan of this paper}
%\LB{Add Section references}

In an infinite-dimensional setting a detailed theory is widely open and will be subject of this paper. We will consider an absolutely one-homogeneous functional  $J$ on a Hilbert space and the corresponding eigenvalue problem $\lambda u \in \partial J(u)$. Effectively this means we look at a semi-norm defined on a subspace (being dense in many typical examples) of the Hilbert space and investigate the associated eigenvalue problem and gradient flow. The basic theory does not assume any relation between $J$ and the norm of the Hilbert space, but we shall see that many favourable properties of the gradient flow -- such as finite time extinction, for instance -- rely on a Poincar\'{e}-type inequality. That is, after factorizing out the null-space of the functional we have a continuous embedding of the Banach space with norm $J$ into the ambient Hilbert space. It is thus natural to think in terms of a Gelfand triple, with subgradients of $J$ existing a-priori only in a dual space which is larger than the Hilbert space. The eigenvalue problem and the gradient flow naturally lead to considering only cases with a subgradient in the Hilbert space, which is an abstract regularity condition known as source condition in inverse problems (cf. \cite{burger2004convergence,benning2018modern}). We shall see that a key role is played by the subgradient of minimal norm (known as minimal norm certificate in compressed sensing and related problems, cf. \cite{fornasier2011compressive,duval2015exact,chambolle2016geometric}). A first key contribution of this paper is a geometric characterization of eigenvectors in such a setting, which is based on a dual characterization of absolutely one-homogeneous convex functionals. Roughly speaking we can interpret all possible subgradients as elements lying on the boundary of a dual unit ball (the subdifferential of $J$ at~$0$) and single out eigenvectors as those elements which are a normal vector to an orthogonal supporting hyperplane of the ball. Thus, the eigenvalue problem becomes a geometric problem for a Banach space relative to a Hilbert space structure. 

We also show that being a subgradient of minimal norm is a necessary condition for eigenvectors. This establishes an interesting connection to gradient flows, which automatically select subgradients of minimal norm as the time derivative of the primal variable. We hence study gradient flows in further detail and conclude that -- if the above-noted geometric condition is satisfied -- they yield a spectral decomposition, i.e., a representation of the initial data $f$ as integral of eigenvectors with decreasing frequency (decreasing Hilbert space norm at fixed dual norm). Moreover, we show that if the gradient flow yields a spectral decomposition, this is already sufficient to obtain equivalence to the variational method as well as an inverse scale space method proposed as an alternative to obtain spectral decompositions (cf. \cite{burger2015spectral}). With an appropriate reparametrization from the time in the gradient flow to a spectral dimension we rigorously obtain a spectral decomposition akin to the spectral decomposition of self adjoint linear operators in Hilbert space as discussed in \cite{burger2016spectral}. We apply our theory to several examples: in particular, it matches the finite-dimensional theory for polyhedral regularizations in \cite{burger2016spectral}, and it can also be used for the one-dimensional total variation flow, a flow of a divergence functional for vector fields, as well as for a combination of divergence and rotation sparsity. Moreover, we visit the simple case of a flow of the $L^1$-norm, which gives further intuition and limitations in a case where no Poincar\'{e}-type estimate between the convex functional and the Hilbert space norm is valid.

Finally, we also discuss the extinction times and extinction profiles of gradient flows, a problem that was studied for TV flow in detail before (cf. \cite{andreu2002some,giga2011scale,bonforte2012total}). Our theory is general enough to allow for a direct generalizations of the results to flows of absolutely one-homogeneous convex functionals and simplifies many proofs. In particular, we can show that under generic conditions the gradient flows have finite extinction time and there is an extinction profile, i.e., a left limit of the time derivative at the extinction, which is an eigenvector. Furthermore, we give sharp upper and lower bounds of the extinction time. For flows that yield a spectral decomposition, we obtain a simple relation between the initial datum, the extinction time, and the extinction profile. In the case of the one-dimensional total variation flow we get the results in \cite{bonforte2012total} as special cases. 

%\subsection{Main contributions}
%\begin{itemize}
%\item Dual characterization of absolutely one-homogeneous convex functionals
%\item Geometric characterization of all eigenvectors of such functionals
%%\item Equivalence of GF, VP, and ISS in case of spectral decomposition
%\item Rigorous definition of a spectral measure that produces the right analogies with the linear spectral measure
%\item Spectral decomposition through polyhedral, $L^1$, $\div$, $\rot$, and $\div+\rot$ flows
%\item Example of effectiveness / simplicity of our theory: extinction profiles of flows (special cases: \cite{andreu2002some} and \cite{bonforte2012total})
%\end{itemize}

The remainder of the paper is organized as follows: in Section~\ref{sec:notation} we discuss some basic properties of absolutely one-homogeneous functionals and the related nonlinear eigenvalue problem, Section~\ref{sec:geo_char} is devoted to obtain further geometric characterizations of eigenvectors and to work out connections to subgradients of minimal norm. In Section~\ref{sec:spec_dec_GF} we discuss the potential to obtain spectral decompositions by gradient flows. For this sake we give an overview of the classical theory by Brezis, which shows that gradient flows generate subgradients of minimal norm, and we also provide equivalence results to other methods in the case of spectral decompositions, for which we give a geometric condition. Moreover, we discuss the appropriate scaling of the spectrum from time to eigenvalues in order to obtain a more suitable decomposition. In Section~\ref{sec:examples} we show that the geometric condition for obtaining a spectral decomposition is indeed satisfied for relevant examples such as total variation flow in one dimension and multidimensional flows of vector fields with functionals based on divergence and rotation. Finally, in Section~\ref{sec:extinction} we investigate the extinction profiles of the gradient flows, which we show to be eigenvectors even if the flow itself does not necessarily produce a spectral decomposition.

\section{Absolutely homogeneous convex functionals and eigenvalue problems}\label{sec:notation}

In the following we collect some results about the basic class of convex functionals we consider in this paper, moreover we provide basic definitions and results about the nonlinear eigenvalue problem related to such functionals. 

\subsection{Notation}
Let $(\H,\langle\cdot,\cdot\rangle)$ be a real Hilbert space with induced norm $\norm{\cdot}:=\sqrt{\langle\cdot,\cdot\rangle}$ and $J$ be a functional in the class $\calC$, defined as follows:
\begin{definition}\label{def:class_C}
The class $\calC$ consists of all maps $J:\H\to\R\cup\{+\infty\}$ such that $J$ is convex, lower semicontinuous with respect to the strong topology on $\H$, and absolutely one-homogeneous, i.e.,
\begin{align}
J(cu)&=|c|J(u),\quad\forall u\in\H,\;c\in\R\setminus\{0\},\label{eq:abs_1-hom_1}\\
J(0)&=0.\label{eq:abs_1-hom_2}
\end{align}
\end{definition}
The effective domain and null-space of $J\in\calC$ are given by
\begin{align}
\dom(J)&:=\{u\in\H\st J(u)<\infty\},\\
\calN(J)&:=\{u\in\H\st J(u)=0\}.
\end{align}
Note that $\dom(J)$ and $\calN(J)$ are not empty due to \eqref{eq:abs_1-hom_2} and that $\calN(J)$ is a (strongly and weakly) closed linear space \cite{bungert2018solution} whose orthogonal complement we denote by $\calN(J)^\perp$. The effective domain $\dom(J)$ is also a linear space but not closed, in general. Given any $f\in\H$, the orthogonal projection $\bar{f}$ of $f$ onto $\calN(J)$ is
\begin{align}\label{eq:orth_proj}
\bar{f}:=\argmin_{u\in\calN(J)}\norm{u-f},
\end{align}
and the ``dual norm'' of $f\in\H$ with respect to $J$ is defined as
\begin{align}\label{eq:dual_norm}
\norm{f}_*:=\sup_{u\in\calN(J)^\perp}\frac{\langle f,u\rangle}{J(u)}.
\end{align}
Considering $J$ as a norm on the Banach space 
\begin{align}\label{eq:Banach_space}
{\cal V} = \dom(J) \cap \calN(J)^\perp,
\end{align}
the above dual norm is indeed defined on the dual space of ${\cal V}$ (respectively a predual if it exists). Hence, we naturally obtain a Gelfand triple structure \change{${\cal V}\hookrightarrow\H\hookrightarrow{\cal V}^*$} and the eigenvalue problem can also be understood as a relation between the geometries of the Hilbert space $\H$ and the Banach spaces ${\cal V}$ and ${\cal V}^*$.

By $\partial J(u)$ we denote the subdifferential of $J$ in $u\in\dom(J)$, given by\secchange{%
\begin{align}
\partial J(u):=\left\lbrace p\in\H\st\langle p,v\rangle\leq J(v),\;\forall v\in\H,\;\langle p,u\rangle=J(u)\right\rbrace,
\end{align}}
and define $\dom(\partial J):=\{u\in\H\st\partial J(u)\neq\emptyset\}.$ Of particular importance will be the subdifferential in zero\secchange{%
\begin{align}\label{eq:subdiff_0}
\partial J(0)=\{p\in\H\st\langle p,v\rangle\leq J(v),\;\forall v\in\H\},
\end{align}}
which uniquely determines $J$ as we will see. Using the definition of the dual norm \eqref{eq:dual_norm}, it can be easily seen that 
\begin{align}
\partial J(0)=\{p\in\calN(J)^\perp\st\norm{p}_*\leq1\},
\end{align}
i.e., roughly speaking the subdifferential of $J$ in ~$0$ coincides with the dual unit ball \change{in the space ${\cal V}^*$}.

Lastly, we recall the definitions of the \emph{Fenchel-conjugate} of a general function $\Phi:\H\to[-\infty,+\infty]$ as $\Phi^*(u):=\sup_{v\in\H}\langle v,u\rangle-\Phi(v)$ and of the \emph{indicator function} of a subset $K\subset\H$ as
$$\chi_K(u):=
\begin{cases}
0,\quad &u\in K,\\
+\infty,\quad &u\notin K.
\end{cases}$$
We refer to \cite{bauschke2011convex} for fundamental properties.

\subsection{Absolutely one-homogeneous functionals}
In the following, we collect some elementary properties of functionals in the class $\calC$ \change{defined in Definition~\ref{def:class_C}} and their subdifferentials whose proofs are either trivial or can be found in~\cite{burger2016spectral}.
\begin{prop}\label{prop:properties_one-hom}
Let $J\in\calC$, $u,v\in\H$, and $c>0$. It holds
\begin{enumerate}
\item $J(u)\geq0$, 
\item $J(u+v)\leq J(u)+J(v)$ and $J(u)-J(v)\leq J(u-v)$,
\item $J(u+w)=J(u)$ for all $w\in\calN(J)$,
\item $J^*(u)=\chi_{\partial J(0)}(u)$,
\item $\partial J(u)$ is convex and weakly closed and it holds $\partial J(u)\subset\partial J(0)$,
\item $\partial J(cu)=\partial J(u)$,
\item $\partial J(u)\perp\calN(J)$, i.e., $\langle p,w\rangle=0$ for all $p\in\partial J(u),\,w\in\calN(J)$.
\end{enumerate}
\end{prop}

As already indicated, the knowledge of the set $\partial J(0)$ suffices to uniquely determine a functional $J\in\calC$. Furthermore, any such functional has a canonical dual representation, similar to the concept of dual norms. This is no surprise since the elements of $\calC$ are semi-norms on subspaces of $\H$ and norms if and only if they have trivial null-space.

\begin{thm}\label{thm:characterization_one-hom}
A functional $J:\H\to\R\cup\{\infty\}$ belongs to $\calC$ if and only if 
$$J(u)=\sup_{p\in K}\langle p,u\rangle=\chi_K^*(u)$$
for a set $K\subset\H$ that meets $K=-K$. In this case, it holds
$$\partial J(0)=\overline{\conv}(K)$$
where $\overline{\conv}$ denotes the closed convex hull of a set.
\end{thm}
\begin{proof}
It is well-known that the Fenchel-conjugate of an absolutely one-homogeneous functional $J$ is given by $J^*(u)=\chi_{\partial J(0)}$. Hence, for the first implication we observe that by choosing $K:=\partial J(0)$ and using that -- being lower semi-continuous and convex -- $J$ equals its double Fenchel-conjugate it holds $J(u)=J^{**}(u)=\chi^*_K(u)=\sup_{p\in K}\langle p,u\rangle.$ Furthermore, $K$ meets $K=-K$ which can be seen from $\langle-p,u\rangle=\langle p,-u\rangle\leq J(-u)=J(u)$ for $p\in K$.

Conversely, any $J$ given by $J(u)=\sup_{p\in K}\langle p,u\rangle=\chi^*_K(u)$ where $K=-K$ trivially belongs to $\calC$ and it holds
$J^*=\chi^{**}_K=\chi_{\overline{\conv}(K)}$. Hence, by standard subdifferential calculus one concludes $\partial J(0)=\{p\in\H\st J^*(p)=0\}=\overline{\conv}(K)$.

\end{proof}

\begin{rem}\label{rem:triangle_integral}
Using the convexity and homogeneity of $J$, Jensen's inequality immediately implies that the generalized triangle inequality 
\begin{align}
J\left(\int_a^b u(t)\d t\right)\leq\int_a^b J(u(t))\d t
\end{align}
for a function $u:[a,b]\to\H$ holds, whenever these expressions make sense.
\end{rem}

Due to Theorem~\ref{thm:characterization_one-hom}, we will from now on assume that 
\begin{align}\label{def:J_and_K}
J(u)=\sup_{p\in K}\langle p,u\rangle,\;u\in\H,
\qquad
K:=\partial J(0),
\end{align}
after possibly replacing $K$ by the closure of its convex hull.

\subsection{Subgradients and eigenvectors}
In this section, we will define the non-linear eigenvalue problem under consideration and provide first insights into the the geometric connection of eigenvectors and the dual unit ball~$K$. We start with general subgradients of a functional $J\in\calC$ before we turn to the special case of subgradients which are eigenvectors.

\begin{definition}[Subgradients]
Let $J\in\calC$ and $u\in\dom(\partial J)$. Then the elements of the set $\partial J(u)$ are called subgradients of $J$ in $u$.
\end{definition}

\begin{prop}\label{prop:subgrad_bdry}
Let $p$ be a subgradient of $J$ in $u\neq 0$ and let $K$ be as in \eqref{def:J_and_K}. Then $p$ lies in the relative boundary $\partial_\mathrm{rel} K$ of $K$, where $\partial_\mathrm{rel} K=K\setminus\mathrm{relint}(K)$ and $\mathrm{relint}(K)=\{p\in K\st \exists c>1: cp\in K\}$ denotes the relative interior of $K$.
\end{prop}
\begin{proof}
We observe that for all $c\in[0,1]$ and $v\in\H$ it holds
$$\langle cp,v\rangle=c\langle p,v\rangle\leq cJ(v)\leq J(v),$$
i.e., $cp\in K$. Let us assume that there is $c>1$ such that $cp\in K$. Then, using that $p\in\partial J(u)$ yields
$$c\langle p,u\rangle=\langle cp,u\rangle\leq J(u)=\langle p,u\rangle$$
which clearly contradicts $c>1$. Hence, we have established the claim
\end{proof}

Interestingly, due to the fact that subdifferentials $\partial J(u)$ are convex sets and lie in the (relative) boundary of the convex set $K$, they are either singletons or lie in a ``flat'' region of the boundary of~$K$.

\begin{definition}[Eigenvectors]
Let $J\in\calC$. We say that $u\in\H$ is an eigenvector of $\partial J$ with eigenvalue $\lambda\in\R$ if
\begin{align}
\lambda u\in\partial J(u).
\end{align}
\end{definition}
\begin{rem}
Due to the positive zero-homogeneity of $\partial J(u)$ any multiple $cu$ of $u$ with $c>0$ is also an eigenvector of $\partial J$ with eigenvalue $\lambda/c$. To avoid this ambiguity one sometimes additionally demands $\norm{u}=1$ from an eigenvector $u$. In this work, however, we do not adopt this normalization since it does not match the flows that we are considering. As a consequence, all occurring eigenvalues should be multiplied by the norm of the eigenvector to become interpretable. E.g., let $p\in\partial J(p)$. Then $q:=p/\norm{p}$ has unit norm and is an eigenvector of $\partial J$ with eigenvalue $\lambda:=\norm{p}$ since $\lambda q=p\in\partial J(p)=\partial J(q)$. The last equality follows from 6. in Proposition~\ref{prop:properties_one-hom}. 
\end{rem}

Now we collect some basic properties of eigenvectors and, in particular, we show that eigenvectors are the elements of minimal norm in their respective subdifferential. Hence, one can restrict the search for eigenvectors to the subgradients of minimal norm; this shows a first connection to gradient flows which select the subgradients of minimal norm, as already indicated.

\begin{prop}[Properties of eigenvectors]\label{prop:prop_eigenvectors}
Let $u\in\H$ be an eigenvector of $\partial J$ with eigenvalue $\lambda\in\R$. Then it holds
\begin{enumerate}
\item $-u$ is eigenvector with eigenvalue $\lambda$,
\item $\lambda\geq0$ and $\lambda=0$ if and only if $u=0$,
\item $\lambda u=\argmin\{\norm{p}\st p\in\partial J(u)\}$.
\end{enumerate}
\end{prop}
\begin{proof}
\secchange{We only proof the third item.} It holds for all $p\in\partial J(u)$ that
$$\lambda\norm{u}^2=\langle\lambda u,u\rangle=J(u)=\langle p,u\rangle\leq\norm{p}\norm{u}$$
and, since $\lambda\geq0$, we obtain $\norm{\lambda u}\leq\norm{p}$. The convexity of $\partial J(u)$ shows that $\lambda u$ is the unique element of minimal norm.
\end{proof}
\begin{rem}
In the following, we will simply talk about eigenvectors whilst suppressing the dependency upon $\partial J$, for brevity. 
\end{rem}

It is trivial that all elements in the null-space of $J$ are eigenvectors with eigenvalue $0$. However, these eigenvectors are only of minor interest as the example of total variation shows, where the null-space consists of all constant functions. Hence, in \cite{benning2013ground} so-called ground states where considered. These are eigenvectors in the orthogonal complement of the null-space with the lowest possible positive eigenvalue and, hence, the second largest eigenvalue of the operator $\partial J$. In our setting, these ground states correspond to vectors with minimal norm in the \change{relative boundary $\partial_\mathrm{rel}K$ of~$K$ which was defined in Proposition~\ref{prop:subgrad_bdry}}. 

\begin{prop}\label{ex:ground_states}
Let $u_0$ be a ground state of $J$, defined as
\begin{align}
&u_0\in\argmin_{\substack{u\in\calN(J)^\perp\\\norm{u}=1}}J(u),\label{eq:ground_state}
\end{align}
and let $\lambda_0:=J(u_0)$. Then $p_0:=\lambda_0 u_0$ is an element of minimal norm in $\partial_\mathrm{rel}K$ and an eigenvector.
\end{prop}
\begin{proof}
It has been shown in \cite{benning2013ground} that ground states $u_0$ in \eqref{eq:ground_state} exist under relatively weak assumptions\footnote{E.g. if $J$ meets \eqref{ineq:general_poincare}} and are eigenvectors with eigenvalue $\lambda_0$. Furthermore, $\lambda_0$ is per definitionem the smallest eigenvalue that a normalized eigenvector in $\calN(J)^\perp$ can have. Hence, $p_0:=\lambda_0u_0\in\partial J(u_0)=\partial J(p_0)$ which shows that $p_0$ is an eigenvector. Let us assume there is $q\in\partial_\mathrm{rel}K$ such that $\norm{q}<\norm{p_0}=\lambda_0$. This implies with the Cauchy-Schwarz inequality and the definition of $\lambda_0$
$$\langle q,u\rangle\leq\norm{q}\norm{u}<\lambda_0\norm{u}=J(u),\quad\forall u\in\calN(J)^\perp.$$
Since this inequality is strict, we can conclude that $q\notin\partial_\mathrm{rel}K$. Therefore, $\norm{p_0}$ is minimal, as claimed. %\LB{Is every minimal vector a ground state?}
\end{proof}

\section{Geometric characterization of eigenvectors}\label{sec:geo_char}

In this section, we give a novel geometric characterizing of eigenvectors. Simply speaking, eigenvectors $p$ (with eigenvalue 1) are exactly those vectors on the relative boundary of $K$ for which there exists a supporting hyperplane of $K$ through $p$ that is orthogonal to $p$. In other words, there is a multiple of the Hilbert unit ball which is tangential to $\partial_\mathrm{rel} K$ at $p$. All other eigenvectors are multiples of these. \change{In particular, the eigenvalue problem can be viewed as studying the \emph{relative geometry} of the unit balls in $\H$ and ${\cal V}^*$, respectively, where $\cal V$ is given by \eqref{eq:Banach_space}.} Since this geometric interpretation is not very handy in case of infinite dimensional Hilbert spaces (e.g. function spaces), we will only work with an algebraic characterization, in the following. We start with a lemma that allows us to limit ourselves to the study of eigenvectors with eigenvalue 1 without loss of generality.

\begin{lemma}
$u\in\H$ is an eigenvector with eigenvalue $\lambda>0$ if and only if $p:=\lambda u$ is an eigenvector with eigenvalue 1. 
\end{lemma}
\begin{proof}
The statement follows directly from 6. in Proposition~\ref{prop:properties_one-hom}.
\end{proof}

A key geometric characterization is provided by the following result: 

\begin{prop}\label{prop:characterization_eigenvectors}
$p\in K$ is an eigenvector with eigenvalue 1 if and only if
\begin{equation}
\langle p, p-q \rangle \geq 0 \qquad \forall q \in K. \label{ineq:eigenvectorcharacterization}
\end{equation}
\end{prop}  
\begin{proof}
It holds that $p$ is an eigenvector with eigenvalue 1 if and only if
$$\langle p,p\rangle=J(p)=\sup_{q\in K}\langle p,q\rangle.$$
This is equivalent to \eqref{ineq:eigenvectorcharacterization} which concludes the proof.
\end{proof}
 
\change{The statement of Proposition~\ref{prop:characterization_eigenvectors} is illustrated in Figure~\ref{fig:eigenvectors} which shows four different dual unit balls $K\subset\R^2$ together with all eigenvectors with eigenvalue one. It is obvious from the picture that any other vector on the boundary does not meet the boundary orthogonally.} 
 
\input{eigenvectors}

Proposition~\ref{prop:characterization_eigenvectors} can be used to obtain the following results:

\begin{cor}
Let $p$ be a point of maximal norm in $K$, i.e. $\Vert p \Vert \geq \Vert q \Vert$ for all $q \in K$. Then every positive multiple of $p$ is an eigenvector. 
\end{cor}

\begin{example}
Consider the linear eigenvalue problem for a symmetric and positively semi-definite matrix $A\in\R^{n\times n}$. The corresponding functional is given by $J(u) = \sqrt{\langle u, A u \rangle}$ and $K$ is an ellipsoid. Along the main axes of the ellipsoid the hypersurface $p+\mathrm{span}\{p^\bot\}$ is tangential, hence the main axes define the eigenvectors. 
\end{example}

\begin{rem}[Existence of nonlinear spectral decompositions]\label{rem:existence_of_eigenbases}
Note that unlike in the above-noted linear case, where there are exactly $n$ different eigendirections, nonlinear eigenvectors in our setting may constitute an overcomplete generating set of the ambient space, as it can be seen in Figure~\ref{fig:eigenvectors}. The leftmost set corresponds to the linear case and has two different eigendirections. \change{In contrast, in the non-linear cases one can have significantly more different eigendirections which makes the set of eigenfunctions an overcomplete system in these examples. Whether this is true in general, remains an open question. However, in Section~\ref{sec:examples} we give several relevant examples where there are sufficiently many nonlinear eigenvectors to represent any datum as linear combination of such vectors. Note that from the second and fourth set in Figure~\ref{fig:eigenvectors} it also becomes clear that in the nonlinear case one cannot expect to have orthogonal eigenvectors as it is the case for compact self-adjoint linear operators.}
\end{rem}

\change{Next we investigate for which elements $q\in K$ the characterizing inequality \eqref{ineq:eigenvectorcharacterization} for eigenvectors is actually an equality.}

\begin{prop}\label{prop:minsub}
Let $u\in\dom(\partial J)$ and ${p}:=\argmin\left\lbrace\norm{q}\st q\in\partial J(u)\right\rbrace$ be an eigenvector with eigenvalue 1. Then it holds
\begin{align}\label{eq:minsub_general}
\langle {p},{p}-q\rangle =0,\quad\forall q\in\partial J(u),
\end{align}
which can be reformulated as $\partial J(u)\subset\partial J({p}).$
\end{prop}
\begin{proof}
The non-negativity of the left-hand side follows directly from the assumption that ${p}$ is an eigenvector and thus fulfills~\eqref{ineq:eigenvectorcharacterization}. For the other inequality, we let $q\in\partial J(u)$ arbitrary and consider $r:=\lambda q+(1-\lambda){p}$ for $\lambda\in(0,1)$, which is in $\partial J(u)$, as well, due to convexity. Using \eqref{ineq:eigenvectorcharacterization} and the minimality of $\norm{{p}}$ yields
\begin{align*}
\lambda\langle{p},q\rangle+(1-\lambda)\norm{{p}}^2=\langle{p},r\rangle\leq\norm{{p}}^2\leq\norm{r}^2=\lambda^2\norm{q}^2+(1-\lambda)^2\norm{{p}}^2+2\lambda(1-\lambda)\langle{p},q\rangle.
\end{align*}
Dividing this by $\lambda(1-\lambda)$ one finds
\begin{align*}
\frac{1}{1-\lambda}\langle{p},q\rangle+\frac{1}{\lambda}\norm{{p}}^2\leq\frac{\lambda}{1-\lambda}\norm{q}^2+\frac{1-\lambda}{\lambda}\norm{{p}}^2+2\langle{p},q\rangle,
\end{align*}
which can be simplified to 
\begin{align*}
\frac{1}{1-\lambda}\langle{p},q\rangle\leq\frac{\lambda}{1-\lambda}\norm{q^2}-\norm{{p}}^2+2\langle{p},q\rangle.
\end{align*}
Letting $\lambda$ tend to zero and reordering shows $\langle {p},{p}-q\rangle\leq0$, hence equality holds.
\end{proof}

The converse statement of  Proposition \ref{prop:minsub} is false in general. This can be seen by choosing $K\subset\R^n$ to be an ellipsoid. In this case all subdifferentials $\partial J(u)$ are singletons since the boundary of an ellipsoid does not contain convex sets consisting of two or more points. Hence, \eqref{eq:minsub_general} is always met but not every boundary point has an orthogonal tangent hyperplane. However, the converse is true in finite dimensions if $K\subset\R^n$ is a polyhedron. \change{In \cite{burger2016spectral} the authors introduced condition \eqref{eq:minsub_general}, which they termed (MINSUB), to study the case of polyhedral functionals $J$, meaning that the set $K$ is a polyhedron. Using (MINSUB) together with another, relatively strong, condition they were able to prove the converse of Proposition~\ref{prop:minsub}, namely that all subgradients of minimal norm are eigenvectors. Below we show that in fact the other condition is superfluous.}

\begin{prop} \label{polyhedraleigenvectors}
Let $K\subset\R^n$ be a convex polyhedron such that for all $u\in\R^n$ the element ${p}:=\argmin\left\lbrace\norm{q}\st q\in\partial J(u)\right\rbrace$ satisfies condition
\begin{align}\label{eq:minsub}
\langle{p},{p}-q\rangle=0,\quad\forall q\in\partial J(u)\tag{MINSUB}
\end{align}
from \cite{burger2016spectral}. Then ${p}$ is an eigenvector.
\end{prop}
\begin{proof}
Let us fix $u\in\R^n$ and let ${p}$ be the element of minimal norm in $\partial J(u)$. By the definition of the subdifferential and by Proposition~\ref{prop:subgrad_bdry} we infer that $\partial J(u)$ -- being the intersection of $K$ and the hypersurface $\{q\in\R^n\st\langle q,u\rangle=J(u)\}$ -- must coincide with a facet $F$ of the polyhedron.  Due to \eqref{eq:minsub}, the set $\mathcal{S}:=\{q\in\R^n\st\langle q,{p}\rangle=\norm{{p}}^2\}$ defines a hypersurface through ${p}$ such that $F\subset\mathcal{S}$ and $\mathcal{S}$ is orthogonal to ${p}$. Since $K$ is convex, all other points in $K$ lie on one side of $\mathcal{S}$ which implies that $\mathcal{S}$ is supporting $K$ and hence $\langle{p},{p}-q\rangle\geq0$ for all $q\in K$. With Proposition~\ref{prop:characterization_eigenvectors} we conclude that ${p}$ is an eigenvector with eigenvalue 1.
\end{proof}
\begin{rem}
Notably, the minimality of $\norm{{p}}$ does not play a role in the proof of Proposition~\ref{polyhedraleigenvectors}. However, from the Cauchy-Schwarz inequality it follows that only subgradients of minimal norm can satisfy \eqref{eq:minsub}.
\end{rem}

\change{Figure~\ref{fig:minsub} shows two polyhedrons together with a subgradient of minimal norm in the subdifferential marked in red. The left polyhedron meets \eqref{eq:minsub} but the right one does not, because it is too distorted. Consequently, for the left polyhedron the subgradient of minimal norm is an eigenvector whereas this is not true for the right one.}

\input{minsub}

\section{Spectral decompositions by gradient flows}\label{sec:spec_dec_GF}

The fact that eigenvectors are subgradients of minimal norms motivates to further study processes that generate such subgradients. Indeed, the theory of maximal monotone evolution equations shows that gradient flows have this desirable property.

\subsection{Gradient flow theory from maximal monotone evolution equations}
In this section we give a concise recap of the theory of non-linear evolution equations due to Brezis \cite{brezis1973ope}, \change{see also \cite{komura1967nonlinear} for an earlier existence result due to K\={o}mura.} The theory deals with the solution of the differential inclusion
\begin{align}\label{eq:general_flow}
\begin{cases}
\partial_t u(t)+Au(t)\ni0,\\
u(0)=f,
\end{cases}
\end{align}
for times $t>0$. Here $A$ denotes a potentially \emph{non-linear} and \emph{multi-valued} operator defined on a subset $\dom(A):=\{u\in\H\st Au\neq\emptyset\}$ and is assumed to be \emph{maximal monotone}. That is,
\begin{align}
\langle p-q,u-v\rangle\geq0,\quad\forall p\in Au,\;q\in Av,
\end{align}
and $A$ cannot be extended to a monotone operator with larger domain (see \cite{brezis1973ope} for a precise definition). Furthermore, one defines
\begin{align}
A^0u:=\argmin\{\norm{p}\st p\in Au\},\;u\in\dom(A),
\end{align}
which is the norm-minimal element in the convex set $Au$.

For these class of operators one has the following
\begin{thm}[Brezis, 1973]\label{thm:brezis1}
For all $f\in\dom(A)$ there exists a unique function $u:[0,\infty)\to\H$ such that
\begin{enumerate}
\item $u(t)\in\dom(A)$ for all $t>0$
\item $u$ is Lipschitz continuous on $[0,\infty)$, i.e., $\partial_t u\in L^\infty(0,\infty;\H)$ (as distributional derivative) and it holds
\begin{align}
\norm{\partial_t u}_{L^\infty(0,\infty;\H)}\leq\norm{A^0f}
\end{align}
\item \eqref{eq:general_flow} holds for almost every $t\in(0,\infty)$
\item $u$ is right-differentiable for \emph{all} $t\in(0,\infty)$ and it holds
\begin{align}
\partial_t^+u(t)+A^0u(t)=0,\quad\forall t\in(0,\infty)
\end{align}
\item The function $t\mapsto A^0u(t)$ is right-continuous and the function $t\mapsto\norm{A^0u(t)}$ is non-increasing
\end{enumerate} 
\end{thm}
\begin{proof}
For the proof see \cite[Theorem~3.1]{brezis1973ope}.
\end{proof}
A important instance of maximally monotone operators are subdifferentials $\partial J$ of lower semi-continuous convex functionals $J:\H\to\R\cup\{+\infty\}$. For these one can relax the assumption $f\in\dom(\partial J)$ to $f\in\overline{\dom(\partial J)}$ and obtains
\begin{thm}[Brezis, 1973]\label{thm:brezis2}
Let $A=\partial J$ where $J:\H\to\R\cup\{+\infty\}$ is lower semi-continuous, convex, and proper, and let $f\in\overline{\dom(A)}$. Then there exists a unique continuous function $u:[0,\infty)\to\H$ with $u(0)=f$ such that
\begin{enumerate}
\item $u(t)\in\dom(A)$ for all $t>0$
\item $u$ is Lipschitz continuous on $[\delta,\infty)$ for all $\delta>0$ and it holds
\begin{align}
\norm{\partial_tu}_{L^\infty(\delta,\infty;\H)}\leq\norm{A^0v}+\frac{1}{\delta}\norm{f-v},\quad\forall v\in\dom(A),\;\forall\delta>0
\end{align}
\item $u$ is right-differentiable for \emph{all} $t\in(0,\infty)$ and it holds
\begin{align}
\partial_t^+u(t)+A^0u(t)=0,\quad\forall t\in(0,\infty)
\end{align}
\item The function $t\mapsto {A^0u(t)}$ is right-continuous for all $t>0$ and the function $t\mapsto\norm{A^0 u(t)}$ is non-increasing
\item The function $t\mapsto J(u(t))$ is convex, non-increasing, Lipschitz continuous on $[\delta,\infty)$ for all $\delta>0$ and it holds
\begin{align}
\frac{\d^+}{\d t} J(u(t))=-\norm{\partial_t^+u(t)}^2,\quad\forall t>0
\end{align}
\end{enumerate}
\end{thm}
\begin{proof}
For the proof see \cite[Theorem~3.2]{brezis1973ope} where it should be noted that right-differentiability of the map $t\mapsto J(u(t))$ follows since it is Lipschitz continuous and non-increasing. 
\end{proof}

Applying Theorem~\ref{thm:brezis2} to the so-called \emph{gradient flow} of the functional $J$
\begin{align}\label{eq:gradient_flow}
\begin{cases}
\partial_t u(t) = - p(t),\\
p\in \partial J(u(t)),\\
u(0)=f,
\end{cases}
\tag{GF}
\end{align}
yields the existence of a unique solution $u:[0,\infty)\to\H$ with associated subgradients $p(t):=-\partial_t^+u(t)\in\partial J(u(t))$ which have minimal norm in $\partial J(u(t))$ for all $t>0$.

\begin{rem}
From now on, we will denote all occurring right-derivatives with the usual derivative symbols $\partial_t$ and $\frac{\d}{\d t}$ to simplify our notation.
\end{rem}

Having the geometric characterization of eigenvectors and the existence theory of gradient flows at hand, we are now interested in the scenario that the gradient flow yields a sequence of subgradients $p(t)$ which are eigenvectors, i.e., $p(t)\in\partial J(p(t))$. Reformulating this using the eigenvector characterization from Proposition~\ref{prop:characterization_eigenvectors} we obtain

\begin{thm}
The gradient flow \eqref{eq:gradient_flow} yields a sequence of eigenvectors $p(t)$ for $t>0$ if and only if $\langle p(t),p(t)-q\rangle\geq0,\quad\forall q\in K,\;\forall t>0.$
\end{thm}

Before giving examples of functionals $J\in\calC$ that guarantee this to happen, we investigate the consequences of such a behavior of the flow. We prove that, in this case, the gradient flow is equivalent to a variational regularization method and an inverse scale space flow. Furthermore, disjoint increments $p(t)-p(s)$ of eigenvectors will turn out to be mutually orthogonal. Finally, we will use the subgradients $p(t)$ of an arbitrary gradient flow to define a measure that acts as generalization of the spectral measure corresponding to a self-adjoint / compact linear operator in the case that $p(t)$ are eigenvectors. 
\subsection{Equivalence of gradient flow, variational method, and inverse scale space flow}\label{sec:equiv}
First we show that if the gradient flow \eqref{eq:gradient_flow} generates eigenvectors, this implies the equivalence with a variational regularization problem \eqref{eq:var_prob}, and the inverse scale space flow \eqref{eq:iss_flow}, given by
\begin{align}
\label{eq:var_prob}\tag{VP}
&\begin{cases}
v(t)=\argmin_{v\in\H} E_t(v)\\
E_t(v)=\frac{1}2 \| v - f\|^2 + t J(v),
\end{cases}\\[10pt]
\label{eq:iss_flow}\tag{ISS}
&\begin{cases}
\partial_\tau r(\tau)=f-w(\tau),\\
r(\tau)\in\partial J(w(\tau)),\\
r(0)=0,\;w(0)=\bar{f}.
\end{cases}
\end{align}
Here $t>0$ denotes a time / regularization parameter, whereas $\tau>0$ will turn out to correspond to the ``inverse time'' $1/t$. Furthermore, the initial condition $w(0)=\bar{f}$ of the inverse scale space flow denotes the orthogonal projection of $f$ onto the null-space of $J$ defined in \eqref{eq:orth_proj} (cf.~\cite{bungert2018solution} for more details). Note that all time derivatives ought to be understood in the weak sense, existent for almost all times, or in the sense of right-derivatives that exist for all times.
\begin{thm}[Equivalence of GF and VP]\label{thm:equivalence_gf_vp}
Let $(u,p)$ be a solution of the gradient flow \eqref{eq:gradient_flow} and assume that for all $t > 0$ it holds $p(t)\in\partial J(p(t))$. Then for $s\leq t$ it holds $p(s) \in \partial J(u(t))$. Moreover, $u(t) = v(t)$ where $v(t)$ solves \eqref{eq:var_prob}.
\end{thm} 
\begin{proof}
From \eqref{ineq:eigenvectorcharacterization} we see that in particular $ \langle p(t), p(t) - p(s) \rangle \geq 0$ holds for all $0<s\leq t$ and hence, using $p(t)=-\partial_t u(t)$ together with 5. in Theorem~\ref{thm:brezis2},
$$  0 \geq \langle -p(t), p(t) - p(s) \rangle = \frac{\d}{\d t} J(u(t)) - \langle \partial_t u(t), p(s) \rangle. $$
Integrating from $s$ to $t$ yields 
$$ J(u(t))- J(u(s)) - \langle u(t),p(s) \rangle + J(u(s)) \leq 0,  $$
which is equivalent to $J(u(t)) \leq \langle u(t),p(s) \rangle$ and hence $p(s) \in \partial J(u(t))$. That $v(t):=u(t)$ solves \eqref{eq:var_prob} follows by observing that the Fejer mean $q(t):=\frac{1}t \int_0^t p(s)~ds$ is the appropriate subgradient for the optimality condition of $v(t)$ being a minimizer of $E_t$, i.e.,
$$v(t)-f+tq(t)=0,\quad q(t)\in\partial J(v(t)).$$
The fact that \eqref{eq:gradient_flow} and \eqref{eq:var_prob} posses unique solutions concludes the proof.
\end{proof} 

Now we prove the equivalence of the variational problems and the inverse scale space flow. To avoid confusion due to the time reparametrization connecting $t$ and $\tau$, we denote the derivatives of $v$ and $q$ with respect to the regularization parameter $t$ in \eqref{eq:var_prob} by $v'$ and $q'$, respectively. For instance, $v'(1/\tau)$ simply means $\left(\partial_t v(t)\right)\vert_{t=1/\tau}$ and this expression exists since by the previous theorem $v=u$ and $u$ is right-differentiable for all $t>0$.
\begin{thm}[Equivalence of VP and ISS]\label{thm:equivalence_vp_iss}
Let the gradient flow \eqref{eq:gradient_flow} generate eigenvectors $p(t)\in\partial J(p(t))$. Let, furthermore, $\{v(t)\st t>0\}$ be the solutions of the variational problem \eqref{eq:var_prob} with subgradients $\{q(t)\st t>0\}$. Then for $\tau:=1/t$ the pair $(w,r)$, given by
\begin{align}
w(\tau)&:=v\left({1}/{\tau}\right)-\frac{1}{\tau}v'\left({1}/{\tau}\right),\\
r(\tau)&:=q\left({1}/{\tau}\right),
\end{align}
is a solution of the inverse scale space flow \eqref{eq:iss_flow}.
\end{thm}
\begin{proof}
From the optimality conditions of \eqref{eq:var_prob} for $t>0$ we deduce $q(t)=[f-v(t)]/t.$ By the quotient rule we obtain $q'(t)=[v(t)-f-tv'(t)]/t^2.$ Inserting $t=1/\tau$ yields $q'(1/\tau)=\tau^2\left[v(1/\tau)-f-\frac{1}{\tau}v'(1/\tau)\right]$. Using this we find with the chain rule
$$\partial_\tau^+ r(\tau)=-\frac{1}{\tau^2}q'(1/\tau)=f-v(1/\tau)+\frac{1}{\tau}v'(1/\tau)=f-w(\tau),$$
hence, $(w,r)$ fulfills the inverse scale space equality. It remains to check that $r(\tau)\in\partial J(w(\tau))$. We use the fact that according to Theorem~\ref{thm:equivalence_gf_vp} the solutions of the gradient flow and the variational problem coincide, i.e, $v=u$ and $v'=u'=-p$, to obtain
$$J(w(\tau))=J\left(v(1/\tau)-\frac{1}{\tau}v'(1/\tau)\right)=J\left(v(1/\tau)+\frac{1}{\tau}p(1/\tau)\right)\leq J(v(1/\tau))+\frac{1}{\tau}J(p(1/\tau)).$$
Using that the subgradients $p(1/\tau)$ are eigenvectors with minimal norm in $\partial J(u(1/\tau))$ and invoking Proposition~\ref{prop:minsub} we infer $J(p(1/\tau))=\norm{p(1/\tau)}^2=\langle q(1/\tau),p(1/\tau)\rangle.$ By inserting this and using that $q(1/\tau)\in\partial J(v(1/\tau))$ we obtain
\begin{align*}
J(w(\tau))&\leq\langle q(1/\tau),v(1/\tau)\rangle+\frac{1}{\tau}\langle q(1/\tau),p(1/\tau)\rangle=\langle r(\tau),w(\tau)\rangle.
\end{align*} 
The fact that $r(\tau)=q(1/\tau)\in K$ finally shows that $r(\tau)\in\partial J(w(\tau))$. Once again, the uniqueness solutions of \eqref{eq:var_prob} and \eqref{eq:iss_flow} concludes the proof.
\end{proof}

\subsection{Orthogonality of the decomposition}\label{sec:orth}
Simple examples of flows with subgradients which are piecewise constant in time show that it is false that two subgradients of a gradient flow corresponding to different time points are orthogonal. However, we are able to show that the differences of subgradients are orthogonal if the subgradients themselves are eigenvectors.  
\begin{thm}\label{thm:orthogonality}
If the gradient flow \eqref{eq:gradient_flow} generates eigenvectors, it holds for $0<r\leq s\leq t$ that 
\begin{align}\label{eq:orthogonality}
\langle p(t),p(s)-p(r)\rangle=0.
\end{align}
\end{thm}
\begin{proof}
As stated in Theorem~\ref{thm:equivalence_gf_vp}, it holds $p(r),p(s),p(t)\in\partial J(u(t))$ with $\norm{p(t)}$ minimal in $\partial J(u(t))$. Hence, the assertion follows directly by employing Proposition~\ref{prop:minsub} with $u:=u(t)$, $\hat{p}:=p(t)$, and $p\in\{p(s),p(r)\}$ to obtain $\langle p(t),p(s)\rangle =\norm{p(t)}^2=\langle p(t),p(r)\rangle.$
\end{proof}

\change{
We already know from Theorem~\ref{thm:equivalence_gf_vp} that if the gradient flow generates eigenvectors it holds $p(s)\in\partial J(u(t))$ for all $s\leq t$. Using Theorem~\ref{thm:orthogonality} above one can even show 
\begin{cor}\label{cor:hierarchy}
If the gradient flow \eqref{eq:gradient_flow} generates eigenvectors, it holds for all $0<s\leq t$ that  
\begin{align}\label{eq:hierarchy_sg}
p(s)\in\partial J(p(t)).
\end{align}
\end{cor}
\begin{proof}
Equation \eqref{eq:orthogonality} implies $\norm{p(t)}^2=\langle p(s),p(t)\rangle$ for all $0\leq s\leq t$. This yields 
$$J(p(t))=\norm{p(t)}^2=\langle p(s),p(t)\rangle\leq J(p(t))$$
and hence $p(s)\in\partial J(p(t))$.
\end{proof}
At this point the most important consequence of Theorem~\ref{thm:orthogonality} is the following corollary, which states that differences of eigenvectors generated by the gradient flow are orthogonal.
}
\begin{cor}
If one defines the \emph{spectral increments}
$$\phi(s,t)=p(t)-p(s),\quad s,t>0,$$
then Theorem~\ref{thm:orthogonality} implies 
$$\left\langle\phi(s_1,t_1),\phi(s_2,t_2)\right\rangle=0$$
for all $0<s_1\leq t_1\leq s_2\leq t_2$.
\end{cor}

Ideally, one would like to obtain this orthogonality relation for the time derivative of $p$ which, however, only exists in a distributional sense. Formally, one defines
$$\phi(t)=-t\partial_tp(t)$$
to obtain a orthogonal spectral representation of the data $f$ in the sense of \cite{burger2016spectral}, i.e.,
$$f-\bar{f}=\int_0^\infty\phi(t)\d t,\quad\langle\phi(s),\phi(t)\rangle=0,\quad s\neq t.$$
Here, $\phi$ formally act as spectral measure. Since, however, this approach fails due to the lacking regularity of $t\mapsto p(t)$, we will present a rigorous definition of a spectral measure in the next section.

\begin{rem}\label{imprem:finite_dim}
It remains to be mentioned that all results in Sections~\ref{sec:equiv} and \ref{sec:orth} have already been proved in a finite dimensional setting \cite{burger2016spectral} and assuming that $K\subset\R^n$ is a polyhedron that meets \eqref{eq:minsub} or the stronger condition that $J(u)=\norm{Au}_1$ with a matrix $A$ such that $AA^T$ is diagonally dominant. However, these results are not stronger than ours because using the results in Section~\ref{sec:geo_char} it is trivial to show that $K$ satisfying \eqref{eq:minsub} implies that the gradient flow yields eigenvectors (see also Theorem~\ref{thm:poly-flow}). Hence, our result are a proper generalization to infinite dimensions and, furthermore, do not require a special structure of the functional $J$ but only that the gradient flow produces eigenvectors.
\end{rem}

\subsection{Large time behavior and the spectral measure}
In this section we aim at defining a measure that corresponds to the spectral measure of linear operator theory in the case that the gradient flow \eqref{eq:gradient_flow} generates eigenvectors. However, all statements in this section are true without this assumption. \change{As already mentioned in the introduction, the gradient flow \eqref{eq:gradient_flow} gives formally rise to a decomposition of an arbitrary datum into subgradients which takes the form
\begin{align}\label{eq:decomposition}
f-\bar{f}=\int_0^\infty p(t)\d t,
\end{align}
where $\bar{f}$ is the null-space component of $f\in\dom(J)$ given by~\eqref{eq:orth_proj}. This decomposition is the basis for constructing the spectral measure, however, up to now it is formal due to the improper integral.}

In order to make \eqref{eq:decomposition} rigorous, we have to investigate the large time behavior of the gradient flow first. \change{First we will show that the gradient flow has a unique strong limit as time tends to infinity which is given by $\bar{f}$. This will allow us to compute
$$\int_0	^\infty p(t)\d t=-\int_0^\infty\partial_t u(t)\d t=u(0)-\lim_{t\to\infty}u(t)=f-\bar{f},$$
which gives the decomposition. From there on we will construct the spectral measure.}

\begin{thm}[Large time behavior]\label{thm:decrease}
Let $u$ solve \eqref{eq:gradient_flow}, with $f\in\dom(J)$. Then $u(t)$ strongly converges to $\bar{f}$ as $t\to\infty$. 
\end{thm}
\begin{proof}
The proof for strong convergence of $u(t)$ to some $u_\infty\in\calN(J)$ as $t\to\infty$ is given in \cite[Theorem~5]{bruck1975asymptotic} for even and hence, in particular, for absolutely one-homogeneous $J$. To see that $u_\infty=\bar{f}$ we observe
$$\langle u(t)-f,v\rangle=\int_0^t\langle\partial_t u(s),v\rangle\d s=-\int_0^t\langle p(s),v\rangle\d s=0,\quad\forall v\in\calN(J),\,t>0,$$
by using 7. in Proposition~\ref{prop:properties_one-hom}. Hence $u(t)-f\in\calN(J)^\perp$ and by using the strong closedness of the orthogonal complement in Hilbert spaces we infer $u_\infty-f\in\calN(J)^\perp$ and $u_\infty\in\calN(J)$. This is equivalent to $u_\infty=\bar{f}$.
\end{proof}

\begin{cor}\label{cor:indef_int}
Let $u$ solve \eqref{eq:gradient_flow}, with $f\in\dom(J)$. Then it holds
\begin{align}
\lim_{T\to\infty}\norm{\int_0^T p(t)\d t-(f-\bar{f})}=0,
\end{align}
\change{which implies \eqref{eq:decomposition}.}
\end{cor}

\change{Finally, the following Lemma implies that one can without loss of generality assume $\bar{f}=0$.}

\begin{lemma}\label{lem:equivalence_init_cond}
The solution of $\partial_t u=-p,\,p\in\partial J(u)$ with $u(0)=f$ is given by $u=v+\bar{f}$ where $v$ solves $\partial_t v=-p,\,p\in\partial J(v)$ with $v(0)=f-\bar{f}$.
\end{lemma}
\begin{proof}
The proof is trivial by using 3. and 7. in Proposition~\ref{prop:properties_one-hom}.
\end{proof}

Now we consider $(u,p)$ solving $\partial_t u=-p$ with $u(0)=f$ and, without loss of generality, we can assume $\bar{f}=0$. This can always be achieved by replacing $f$ with $f-\bar{f}$ and using the previous Lemma. Note that $\bar{f}=0$ if and only if $f\in\calN(J)^\perp$. Then by Corollary~\ref{cor:indef_int} we infer that
\begin{align}\label{eq:reconstruction}
f=\int_0^\infty p(t)\d t.
\end{align}
Let us compare this to the statement of the spectral theorem for a self-adjoint linear operator $T:\H\to\H$. For these one has
$$\Phi(T)f=\int_{\sigma(T)}\Phi(\lambda)\d E_\lambda f,$$
where $\sigma(T)$ denotes the spectrum of $T$, $E$ is the spectral measure, $\Phi:\R\to\R$ is a function, and $f\in\H$. By choosing $\Phi=\id$ one obtains the spectral decomposition of the operator $T$ itself and by choosing $\Phi=1$ one obtains the decomposition of the identity instead:
\begin{align}\label{eq:res_of_id_linear}
f=\int_{\sigma(T)}\d E_\lambda f.
\end{align}
If $T$ is even compact, the spectral measure becomes atomic and is given by 
\begin{align}\label{eq:atomic_spectral_measure}
E_\lambda=\sum_{k=1}^\infty\delta_{\lambda_k}(\lambda)\langle\cdot,e_k\rangle e_k
\end{align}
where $(e_k)_{k\in\N}$ denotes a set of orthonormal eigenvectors of $T$ with a null-sequence of eigenvalues $(\lambda_k)_{k\in\N}$. Plugging this in, one can express any $f\in\H$ by a linear combination of eigenvectors:
\begin{align}\label{eq:sum_of_ev_linear}
f=\sum_{k=1}^\infty\langle f,e_k\rangle e_k.
\end{align}
Consequently, our aim is to manipulate the measure $p(t)\d t$ in \eqref{eq:reconstruction} in such a way that it becomes a non-linear generalization of \eqref{eq:res_of_id_linear}-\eqref{eq:sum_of_ev_linear} for the case of the maximal monotone operator $\partial J$. In particular, it should become atomic if $t\mapsto p(t)$ is a sequence of countably / finitely many distinct eigenvectors with eigenvalue 1. To achieve this we condense all $t>0$ with the same value of $\norm{p(t)}$ into one atomic point $\lambda(t)=\norm{p(t)}$ which can be considered as the corresponding eigenvalue of $e(t):=p(t)/\norm{p(t)}$. Since the function $t\mapsto\lambda(t)$ is non-increasing and converges to zero according to Theorem~\ref{thm:brezis2} and \cite[Theorem~7]{brezis1974monotone}, this yields a perfect analogy to the linear situation where only orthogonality has to be replaced by orthogonality of differences of eigenvectors. 

\begin{definition}[Spectral measure]
Let $(u,p)$ solve $\partial_tu=-p,\;p\in\partial J(u)$ with $u(0)=f\in\calN(J)^\perp\cap\dom(J)$ and and let $\tilde{\mu}$ be the measure
\begin{align}
\tilde{\mu}(A)=\int_A p(t)\d t,
\end{align}
for Borel sets $A\subset(0,\infty)$. If we set $\lambda:(0,\infty)\to[0,\infty),\;t\mapsto\lambda(t):=\norm{p(t)}$, then the \emph{spectral measure} $\mu$ of $f$ with respect to $J$ is defined as the pushforward of $\tilde{\mu}$ through $\lambda$, i.e., 
\begin{align}
\mu(B):=\tilde{\mu}(\lambda^{-1}B),
\end{align} 
for Borel sets $B\subset[0,\infty)$. The \emph{spectrum} of $\mu$ is given by
\begin{align}
\sigma(\mu):=\{\lambda(t)\st t>0\}.
\end{align}
\end{definition}

\begin{rem}
Note that $t\mapsto \lambda(t)$ is indeed a measurable map since it is non-increasing according to 4. in Theorem~\ref{thm:brezis2} which makes $\mu$ well-defined. 
\end{rem}

By definition of $\mu$ it holds
$$\int_{\sigma(\mu)}\d\mu=\int_{(0,\infty)}\d\tilde{\mu}=f,$$
i.e., $\mu$ has a reconstruction property like \eqref{eq:res_of_id_linear}. Furthermore, if $t\mapsto p(t)$ is a collection of countably many distinct eigenvectors, the map $t\mapsto\lambda(t)$ only has a countable range. Consequently, the measure $\mu$ -- which is supported on the range of $\lambda$ -- becomes concentrated in countably many points and, hence, atomic. This is the analogy to the linear case \eqref{eq:atomic_spectral_measure}. 
\subsection{A \change{necessary and} sufficient condition for spectral decompositions}
Before moving to examples of specific functionals whose gradient flows yield spectral decompositions of \emph{any data}, we first give a necessary and sufficient condition on the data such that the gradient flow of \emph{any functional} computes a spectral decomposition of the data. \change{More precisely, we show that the necessary condition \eqref{eq:hierarchy_sg} derived in Corollary~\ref{cor:hierarchy} is also sufficient. This condition generalizes} the (SUB0) + orthogonality condition defined in \cite{schmidt2018inverse}, which appears to be the only sufficient condition in that line.

\begin{thm}\label{thm:suff_cond_decomp}
Let $T>0$ and assume that 
\begin{align}\label{eq:nec_suff_cond_dec}
f=\int_0^\infty p(s)\d s,\quad p(s)\in\partial J(p(t)),\quad\forall 0< s\leq t.
\end{align}
Then $u(t)=\int_t^\infty p(s)\d s$ solves the gradient flow \eqref{eq:gradient_flow}, i.e., $\partial_t u(t)=-p(t)\in\partial J(u(t))$ for $t>0$.
\end{thm}
\begin{proof}
Obviously, it holds $\partial_tu(t)=-p(t)$ and thus it remains to be checked that $p(t)\in\partial J(u(t))$ for $t>0$. We compute 
$$J(u(t))=J\left(\int_t^\infty p(s)\d s\right)\leq\int_t^\infty J(p(s))\d s=\int_t^T\langle p(t),p(s)\rangle\d s=\langle p(t),u(t)\rangle.$$
Together with $p(t)\in K$ for all $t>0$ this concludes the proof.
\end{proof}
\change{
\begin{definition}[(SUB0) + orthogonality \cite{schmidt2018inverse}]\label{def:sub0}
We say that $f\in\H$ meets the (SUB0) + orthogonality condition if there are $N\in\N$ and positive numbers $\gamma_i,\,\lambda_i$ for $i=1,\dots,N$ such that 
$$f=\sum_{i=1}^N\gamma_i u_i$$
where $p_i:=\lambda_i u_i$ meet
\begin{itemize}
\item $p_i\in\partial J(p_i)$ for $i=1,\dots,N$,
\item $\sum_{i=j}^N p_i\in \partial J(0)$ for all $j=1,\dots,N$,
\item $\langle p_i, p_j\rangle = 0$ for $i,j=1,\dots,N$ with $i\neq j$.
\end{itemize}
\end{definition}
In \cite{schmidt2018inverse} the authors proved that the inverse scale space flow \eqref{eq:iss_flow} yields a decomposition into eigenvectors if the datum meets the conditions above. In the following we prove that their condition is a special case of our necessary and sufficient condition \eqref{eq:nec_suff_cond_dec}. 
}

\begin{thm}
Any datum $f$ fulfilling the conditions from Definition~\ref{def:sub0} also meets \eqref{eq:nec_suff_cond_dec}.
\end{thm}
\begin{proof}
Any finite representation of the data as
$$f=\sum_{i=1}^N\gamma_i u_i$$
with numbers $\gamma_i\geq 0$ and eigenvectors $u_i$ meeting $\lambda_i u_i\in\partial J(u_i)$ can be rewritten as
$$f=t_1\sum_{i=1}^N\lambda_iu_i+(t_2-t_1)\sum_{i=2}^N\lambda_iu_i+\dots+(t_N-t_{N-1})\sum_{i=N}^N\lambda_iu_i$$
where $t_i:=\gamma_i/\lambda_i$ and we can assume the ordering $t_i<t_j$ for $i<j$. Consequently, we can define
$$p(s):=\sum_{i=j}^N\lambda_iu_i,\;s\in[t_{j-1},t_j]$$
for $i=1,\dots,N$ and $t_0:=0$. On one hand, this yields $f=\int_0^Tp(s)\d s$ where $T:=t_N$. On the other hand, using \cite[Proposition~3.4]{schmidt2018inverse} one can easily calculate that $p(s)\in\partial J(p(t))$ holds for $s\leq t$, which is condition \eqref{eq:nec_suff_cond_dec}.
\end{proof}

\subsection{Short time behavior}
Finally, we investigate the consequences of a gradient flow which generates eigenvectors for the short time behavior of the map $t\mapsto u(t)$. From the temporal continuity it is obvious that $u(t)\to u(0)=f$ in $\H$ as $t\searrow 0$ which holds in general. However, in the case that $f\in\dom(J)$ and the gradient flow yields a spectral decomposition into eigenvectors more can be shown. In fact, $u(t)$ converges to $f$ also in the semi-norm $J$ which generalizes known results for the one-dimensional ROF problem~\cite{overgaard2017taut}.
\begin{thm}
Let $f\in\dom(J)$ and assume that the gradient flow \eqref{eq:gradient_flow} generates eigenvectors. Then it holds $J(f-u(t))\to 0$ as $t\searrow 0$.
\end{thm} 
\begin{proof}
Using the definition of the gradient flow, we write $f-u(t)=\int_0^t p(s)\d s$. Since all the subgradients $p(s)$ for $s>0$ are eigenvectors, this implies together with the triangle inequality from Remark~\ref{rem:triangle_integral}
$$0\leq J(f-u(t))\leq\int_0^tJ(p(s))\d s=\int_0^t\norm{p(s)}^2\d s.$$
According to \cite{brezis1971proprietes} the map $s\mapsto\norm{p(s)}$ is in $L^2(0,\delta)$ for all $\delta>0$ if and only if $f\in\dom(J)$. Hence, the dominated convergence theorem yields $\lim_{t\searrow 0}J(f-u(t))=0$, as desired.
\end{proof}

\section{Examples of flows that yield a spectral decomposition}\label{sec:examples}

In the following we discuss some relevant examples of functionals $J$ and corresponding flows that yield a spectral decomposition, \change{meaning that the decomposition
$$f-\bar{f}=\int_0^\infty p(t)\d t$$
induced by the gradient flow is a decomposition into eigenvectors $p(t)\in\partial J(p(t))$ for all $t>0$.}

\subsection{Polyhedral flow with \eqref{eq:minsub}}\label{sec:poly_flow}
As already indicated, we can prove that condition \eqref{eq:minsub} \change{introduced in Proposition~\ref{polyhedraleigenvectors}} is already sufficient for the gradient flow yielding eigenvectors, which is an improvement of the results in \cite{burger2016spectral}.
\begin{thm}[Polyhedral gradient flow]\label{thm:poly-flow}
Let $K=\partial J(0)$ be a polyhedron satisfying \eqref{eq:minsub} \change{from Proposition~\ref{polyhedraleigenvectors}}. Then the gradient flow $\partial_t u = - p ,\;p\in \partial J(u)$ with $u(0)=f$ yields a sequence of eigenvectors $p(t)$ for $t>0$. Furthermore, $t\mapsto p(t)$ takes only finitely many different values and there is $T>0$ such that $p(t)=0$ for all $t\geq T$.
\end{thm} 
\begin{proof}
The gradient flow selects subgradients of minimal norm (see Theorem~\ref{thm:brezis2}), all of them are eigenvectors due to Proposition \ref{polyhedraleigenvectors}. According to \cite{burger2016spectral} we also know that $p$ is piecewise constant in~$t$ and that the flow becomes extinct in finite time.
\end{proof}

As mentioned in Remark~\ref{imprem:finite_dim}, in finite dimensions functionals of the type $J(u)=\norm{Au}_1$ with $AA^T$ diagonally dominant satisfy \eqref{eq:minsub} and thus suffice for the gradient flow to yield eigenvectors. Therefore, we will now study functionals of the type $J(u)=\norm{Au}_1$ on $\H=L^2$ where $A$ is a suitable operator such that formally $AA'$ is diagonally dominant. Here $A'$ denotes the adjoint of $A$ with respect to the inner product on $L^2$. Apart from the trivial choice $A=\id$ which will be shortly dealt with in the following section, a natural choice is $A=\div$ since then formally $AA'=\div\nabla=-\Delta$ which is diagonally dominant. Indeed, this does the trick as we will see. Note that in one space dimension, $A$ reduces to the usual derivation operator which means that $J(u)=\norm{u'}_1$ formally becomes the total variation of $u$. Since, compared to the one-dimensional case, the proof that the subgradients of minimal norm are eigenvectors is much more technically involved in two or more space dimensions, we give the 1D proof first to illustrate the ideas before proving the statement in arbitrary space dimension in the subsequent section.

\subsection{$L^1$-flow}\label{sec:l1flow}
We consider $\H=L^2(\Omega)$ and
\begin{align}\label{eq:l1norm}
J(u)=\begin{cases}
\norm{u}_{L^1(\Omega)},\quad&u\in L^1(\Omega)\cap L^2(\Omega),\\
+\infty,\quad&\text{else}.
\end{cases}
\end{align}
Then it can be easily seen that $K=\{p\in L^2(\Omega)\st \norm{p}_{L^\infty}\leq 1\}$ and for any $u\in\dom(\partial J)$ we have
$$\partial J(u)=\left\{p\in K\st\int_\Omega pu\dx=\int_\Omega|u|\dx\right\},$$
i.e., $p(x)=u(x)/|u(x)|$ for $u(x)\neq 0$ and $p(x)\in[-1,1]$ else. The subgradient of minimal norm in $\partial J(u)$ fulfills $p(x)=0$ in the latter case.

\begin{prop}
Let $u\in\dom(\partial J)$ and let $p\in\partial J(u)$ be the subgradient of minimal norm. Then it holds $\langle p,p-q\rangle\geq0$ for all $q\in K.$
\end{prop}
\begin{proof}
We calculate with Cauchy-Schwarz
$$\langle p,p-q\rangle=\int_\Omega|p|^2\dx-\int_\Omega pq\dx\geq\int_\Omega|p|^2\dx-\int_\Omega |p|\dx\geq0$$
since $p$ only takes the values $0$, $-1$, or $+1$ and thus satisfies $|p|=|p|^2$.
\end{proof}

Thus, we obtain the following result:
\begin{thm}[$L^1$-flow]
Let $J$ be given by \eqref{eq:l1norm}. Then the gradient flow $\partial_tu(t)=-p(t),\;p(t)\in\partial J(u(t))$ with $u(0)=f$ generates as sequence of eigenvectors $p(t)$ for~$t>0$.
\end{thm}

Of course this example is of limited practical relevance, but can yield some insight since we can explicitly compute the solutions of the gradient flow noticing that the subgradient just equals the sign of $u$ pointwise, hence splitting $f = f_+ - f_{-}$ with nonnegative functions $f_+$ and $f_-$ of disjoint support we have
$$ u(x,t) =   (f_+(x) - t)_+ - (f_-(x) - t)_+ .$$

\change{
\subsection{$L^\infty$-flow}\label{sec:linfflow}
As before we consider $\H=L^2(\Omega)$ and define
\begin{align}\label{eq:linfnorm}
J(u)=\begin{cases}
\norm{u}_{L^\infty(\Omega)},\quad&u\in L^1(\Omega)\cap L^\infty(\Omega),\\
+\infty,\quad&\text{else}.
\end{cases}
\end{align}
It is obvious that $K=\{p\in L^2(\Omega)\st\norm{p}_{L^1(\Omega)}\leq 1\}$ and in order to study $\partial J(u)$ we need to define the \secchange{set of points where $|u|$ attains its essential maximum:
\begin{align}
\Omega_{\max}:=\{x\in\Omega\st |u(x)|=J(u)\},
\end{align}
which is defined up to Lebesgue measure zero. It is easy to see that for every $u\in\dom(\partial J)$ it holds
\begin{align*}
\partial J(u)=\left\lbrace p\in L^2(\Omega)\st \norm{p}_{L^1(\Omega)}=1,\;p=0\text{ a.e. in }\Omega\setminus\Omega_{\max},\; \sgn(p)=\sgn(u)\text{ a.e. in }{\Omega}_{\max}\right\rbrace.
\end{align*}}
Note that the subgradient of minimal norm is given by 
$$p(x)=
\begin{cases}
\frac{\sgn(u)}{|{\Omega}_{\max}|},\quad&x\in{\Omega}_{\max},\\
0,\quad&\text{else,}
\end{cases}
$$
which follows from 
$$1=\norm{q}_{L^1(\Omega)}^2\leq|\Omega_{\max}|\norm{q}_{L^2(\Omega)}^2,\quad\forall q\in\partial J(u)$$
together with the fact that $\norm{p}_{L^2(\Omega)}^2=1/|\Omega_{\max}|$. 
\begin{prop}
Let $u\in\dom(\partial J)$ and let $p\in\partial J(u)$ be the subgradient of minimal norm. Then it holds $\langle p,p-q\rangle\geq0$ for all $q\in K.$
\end{prop}
\begin{proof}
Using $\norm{q}_{L^1(\Omega)}\leq 1$ for $q\in K$ and the explicit form of the subgradient of minimal norm, we calculate
\begin{align*}
\langle p,p-q\rangle&=\norm{p}_{L^2(\Omega)}^2-\frac{1}{|\Omega_{\max}|}\int_{\Omega_{\max}}\sgn(u)q\dx\geq \frac{1}{|\Omega_{\max}|}-\frac{1}{|\Omega_{\max}|}\int_{\Omega_{\max}}|q|\dx\geq 0.
\end{align*}
\end{proof}
Just as before, we obtain the following result:
\begin{thm}[$L^\infty$-flow]
Let $J$ be given by \eqref{eq:linfnorm}. Then the gradient flow $\partial_tu(t)=-p(t),\;p(t)\in\partial J(u(t))$ with $u(0)=f$ generates a sequence of eigenvectors $p(t)$ for~$t>0$.
\end{thm}}

\subsection{1D total variation flow}\label{sec:1D-TV-flow}
Let $I \subset \R$ be a bounded interval and $J$ be the total variation defined on $L^2(I)$, i.e. extended by infinity outside $BV(I)$. From \cite{briani2011gradient} we infer that $K=\{-g'\st g\in H^1_0(I,[-1,1])\}$ and
\begin{align}\label{eq:subdiff_1d_tv}
\partial J(u)=\{-g'\st g\in H^1_0(I,[-1,1]),\;g\equiv\pm 1\text{ on }\supp((Du)_\pm)\}
\end{align}
where $Du=(Du)_+-(Du)_-$ is the Jordan decomposition of $Du$. We start with some results further characterizing $\dom(\partial J)$ and subgradients of minimal norm in this case:
\begin{lemma} 
Let $u \in \dom(\partial J)$ and $p\in\partial J(u)\setminus\{0\}$. Moreover, let $Du = (Du)_+ - (Du)_-$ be the Jordan decomposition of $Du$. Then for each $x_\pm \in\supp((Du)_\pm)$ the estimate
\begin{equation}
\vert x_+ - x_- \vert \geq \frac{4}{ \Vert p \Vert_{L^2}^2} \label{xdifference}
\end{equation}
holds. Moreover, the distance of $x_\pm$ from $\partial  I$ is bounded below by $\frac{1}{ \Vert p \Vert_{L^2}^2}$.
\end{lemma}
\begin{proof}
We write $p=-g'$ according to \eqref{eq:subdiff_1d_tv}. Thus, for $x_\pm \in \supp((Du)_\pm)$ we find with the Cauchy-Schwarz inequality
$$ 2= \vert g(x_+) - g(x_-) \vert = \left\vert \int_{x_-}^{x_+} g'(x)\dx \right\vert \leq \sqrt{\vert  x_+ - x_- \vert } ~\Vert p \Vert_{L^2}, $$
which yields \eqref{xdifference}. The estimate for the distance to the boundary follows by an analogous argument noticing that $g$ has zero boundary values.
 \end{proof}

\begin{lemma}
Let $u \in \dom(\partial J)$ and $p=-g'$ be the element of minimal $L^2$-norm in $\partial J(u)$. 
Then $g$ is a piecewise linear spline with zero boundary values and a finite number of kinks where $g$ attains the values $+1$ or $-1$.
\end{lemma}
\begin{proof}
First of all, due to \eqref{xdifference} there can only be a finite number of changes of $g$ between $+1$ and $-1$. 
Let $y,z$ be two points in $I$ such that $g(y)=g(z)=1$ and $(y,z) \cap\supp(Du)_-=\emptyset$. Let 
$$ \tilde g(x) = \begin{cases} \max(g(x),1), & \text{if } x \in (y,z), \\ g(x), & \text{else.} \end{cases}$$
Then $-\tilde g' \in \partial J(u)$ and 
$$ \int_I (\tilde g')^2 \dx- \int_I (g')^2 \dx = - \int_y^z (g')^2\dx \leq 0, $$
with equality if and only if $p=-g'$ vanishes in $(y,z)$. Due to the minimality of $g'$ we find that $g\equiv 1$ on $(y,z)$. By an analogous argument we can show that $g \equiv -1$ on each interval $(y,z)$ such that $g(y)=g(z) =-1$  and $(y,z) \cap\supp((Du)_+)=\emptyset$. Now let $y,z \in I$ be such that $g(y)=1$, $g(z)=-1$ and $(y,z) \cap \supp(Du)=\emptyset$. Then we can define
$$ \tilde g(x) =\begin{cases}  g(y) + \frac{x-y}{z-y} (g(z) - g(y)), & \text{if } x \in (y,z), \\ g(x), & \text{else,} \end{cases} $$
and see again that 
$$ \int_I (\tilde g')^2 \dx \leq \int_I (g')^2 \dx.$$
A similar argument shows that $g$ is piecewise linear close to the boundary of $I$, changing from $0$ to $+1$ or $-1$, which completes the assertion.
\end{proof}

\begin{lemma}
Let $u \in \dom(\partial J)$ and $p$ be the element of minimal $L^2$-norm in $\partial J(u)$. Then $	\langle p, p - q \rangle \geq 0$ for all $q \in K.$
\end{lemma}
\begin{proof}
Let the jump set of $p = - g'$ be $\{x_i\}_{i=1}^N$ in ascending order. Then $g \in H_0^1(I,[-1,1])$ satisfies either $g(x)=1$ or $g(x)= -1$ in $x=x_i$ and is piecewise linear  in between. Denoting by $x_0$ and $x_{N+1}$ the boundary points of $I$ we can use $g(x_0)=g(x_{N+1}) =0$ and have $g$ as the piecewise linear spline between all the $x_i$. Now we can write any $q \in \partial J(0)$ as $q=-h'$ for some $h \in H_0^1(I,[-1,1])$. Thus, we have
$$ \langle p, p - q \rangle = \int_I g'(g'-h')\dx  = \sum_{i=0}^N \int_{x_i}^{x_{i+1}} g' (g'-h')\dx.$$
Since $g'$ is piecewise constant we find
$$ \int_{x_i}^{x_{i+1}} g' (g'-h')\dx = \frac{g(x_{i+1})-g(x_i)}{x_{i+1} - x_i} (g(x_{i+1})-g(x_i)-h(x_{i+1})+h(x_i)) $$ 
If $g(x_{i+1})=g(x_i)$, then the integral vanishes. If $0<i<N$ and $g(x_{i+1})\neq g(x_i)$, then
\begin{eqnarray*}\int_{x_i}^{x_{i+1}} g' (g'-h')~dx &=&\frac{4}{x_{i+1} - x_i} - \frac{g(x_{i+1})-g(x_i)}{x_{i+1} - x_i}(h(x_{i+1})-h(x_i)) \\ &\geq& \frac{4 - 2 \vert h(x_{i+1})-h(x_i) \vert}{x_{i+1} - x_i} \\ &\geq& 4\frac{1-  \Vert h \Vert_\infty}{x_{i+1} - x_i} \geq 0. \end{eqnarray*}
A similar argument for the boundary terms $i=0$ and $i=N$ finally yields
$$ \langle p, p - q \rangle \geq 0. $$
\end{proof}

\begin{rem}
Analogously, similar computations can be made in case of an unbounded interval $I$ or the case $I=\R$ where the subgradients are given by $p=-g'$ where $|g(x)|\leq 1$ and $g'\in L^2(\R)$.
\end{rem}

\begin{thm}[One-dimensional total variation flow]
Let $J$ be given by the one-dimensional total variation on an interval $I\subset\R$. Then the gradient flow $\partial_t u(t) = - p(t),\;p(t) \in \partial J(u(t))$ with $u(0)=f$
yields a sequence of eigenvectors $p(t)$ for $t>0$.
\end{thm} 

\begin{rem}[Geometric structure of eigenvectors]
From the characterization of the sub-differential \eqref{eq:subdiff_1d_tv} of the one-dimensional total variation, we conclude that $p$ is an eigenvector if and only if $p=-g'$ where $g\equiv\pm 1$ on $\supp((-g'')_\pm)$. In particular, this implies that $p$ is a step function which jumps at the kinks of $g$.
\end{rem}

\subsection{Divergence flow}

Let $n{\geq 2}$ and $\Omega\subset\R^n$ be an open and bounded set. We consider the gradient flow
\begin{align}\label{eq:grad_flow_div}
\partial_t u=-p,\;p\in\partial J(u),
\end{align}
where 
\begin{align}\label{eq:divergence}
J(u):=\int_\Omega|\div u|:=\sup\left\lbrace-\int_\Omega u\cdot\nabla v\dx\st v\in\testSpaceV\right\rbrace,\;u\in L^2(\Omega,\R^n),
\end{align}
and $u(0)\in L^2(\Omega,\R^n)$. It holds that $u\in\dom(J)$, i.e., $J(u)<\infty$, if and only if the distribution $\div u$ can be represented as a finite Radon measure (cf.~\cite{briani2011gradient}). The null-space $\calN(J)$ of $J$ is infinite dimensional since it contains all vector fields of the type $\nabla v+w$ where $v$ is an harmonic function and $w$ is a divergence-free vector field. Note that $J$ can be cast into the canonical form $J=\chi_K^*$ by setting
\begin{align}\label{eq:set_K_div}
K:=\left\lbrace-\nabla v\st v\in\spaceV\right\rbrace.
\end{align}
Hence, a subgradient $p$ of $J$ in $u$ fulfills $p\in K$ and $\int_\Omega p\cdot u=\int_\Omega|\div u|\dx$. To understand the meaning of $v$ in \eqref{eq:set_K_div}, we perform a integration by parts for smooth $u$ to obtain
$$\int_\Omega p\cdot u\dx=-\int_\Omega\nabla v\cdot u\dx=\int_\Omega v\,\div u\dx.$$
Therefore, $v$ should be chosen as 
\begin{align}\label{eq:v_smooth}
v(x)\in 
\begin{cases}
\{1\},\quad&\div u(x)>0,\\
\{-1\},\quad&\div u(x)<0,\\
[-1,1],\quad&\div u(x)=0,
\end{cases}\quad x\in\Omega.
\end{align}
In the general case, one considers the polar decomposition of the measure $\div u$, given by $\div u=\theta_u\,|\div u|,$ where $\theta_u(x):=\frac{\div u}{|\div u|}(x)$ denotes the Radon-Nikodym derivative, which exists $|\div u|$-almost everywhere and has values in $\{-1,1\}$. This allows us to define the sets
\begin{align}
\E_u^\pm:=\{x\in\Omega\st\theta_u(x)\text{ exists and }\theta_u(x)=\pm 1\}.
\end{align}
Note that trivially it holds $\E_u^+\cup\E_u^-\subset\supp(\div u)$  and one can easily prove that $\supp(\div u)=\overline{\E_u^+\cup\E_u^-}$. In \cite{briani2011gradient} the authors showed that, in full analogy to \eqref{eq:v_smooth}, the subdifferential of $J$ in $u\in\dom(J)$ can be characterized as
\begin{align}\label{eq:subdiff}
\partial J(u)=\left\{-\nabla v\st v\in\spaceV,\,v=\pm1\;|\div u|\text{-a.e. on }\E_u^\pm\right\}.
\end{align}
If $p=-\nabla v$ is a subgradient of $J$ in $u$ we refer to $v$ as a \emph{calibration} of $u$.

\begin{rem}
Since $v\in H^1_0(\Omega,[-1,1])$ can be defined pointwise everywhere except from a set with zero $H^1$-capacity and the measure $\div u$ vanishes on such sets (cf.\cite[Ch.~6]{adams2012function} and \cite{chen2009gauss}, respectively), the pointwise definition of $v$ in~\eqref{eq:subdiff} makes sense.
\end{rem}

Since the gradient flow selects the subgradients with minimal norm, we are specifically interested
in
\begin{align*}
p^0_u=\argmin\left\{\int_\Omega|p|^2\dx\st p\in\partial J(u)\right\}.
\end{align*}
Using \eqref{eq:subdiff} one has
\begin{align}\label{eq:minimal_subgradient}
p^0_u=-\nabla\argmin\left\{\int_\Omega|\nabla v|^2\dx\st v\in\spaceV,\,v=\pm1\;|\div u|\text{-a.e. on }\E_u^\pm\right\},
\end{align}
which implies that $p^0_u=-\nabla v$ where $v$ minimizes the Dirichlet energy and, in particular, is harmonic on the open set $\Omega_H:=\Omega\setminus(\overline{\E_u^+\cup\E_u^-})$. 

Now we show that the gradient flow \eqref{eq:grad_flow_div} indeed generates eigenvectors. We have the following  

\begin{prop}\label{prop:eigenvectors_div}
Let $u\in\dom(\partial J)$ and $p=p_u^0$. Then $p\in\partial J(p)$.
\end{prop}

To prove the proposition, by means of Proposition~\ref{prop:characterization_eigenvectors} it  suffices to check that $\langle p,p-q\rangle\geq0$ for all $q\in K$. We start with an approximation Lemma.

\begin{lemma}\label{lem:approx}
Let $0<\eps<1$ and define
\begin{align}
\psi_\eps(t)=
\begin{cases}
0,&|t|>1,\\
1,&|t|<1-\eps,\\
\frac{1}{\eps}(1-|t|),&1-\eps\leq|t|\leq1,
\end{cases}
\qquad\phi_\eps(t)=\int_0^t\psi_\eps(\tau)\d\tau.
\end{align}
If $v\in H^1_0(\Omega,[-1,1])$ then $v_\eps:=\phi_\eps\circ v\in H^1_0(\Omega,[-1,1])$ and $v_\eps\to v$ strongly in $H^1$ as $\eps\searrow0$. 
\end{lemma}
\begin{proof}
The membership to $H^1_0(\Omega,[-1,1])$ follows directly from the chain rule for compositions of Lipschitz and Sobolev functions. For strong convergence it suffices to show $\nabla v_\eps\to\nabla v$ in $L^2(\Omega)$. Using $\nabla v_\eps=\psi_\eps(v)\nabla v$ this follows from
\begin{align*}
\norm{\nabla v_\eps-\nabla v}_{L^2}^2=&\norm{(\psi_\eps(v)-1)\nabla v}_{L^2}^2=\int_{\{1-\eps\leq|v|\leq 1\}}\underbrace{\left[\frac{1-|v(x)|}{\eps}\right]^2}_{\leq1}|\nabla v|^2\dx\\
\leq&\int_{\{1-\eps\leq|v|\leq1\}}|\nabla v|^2\dx\to\int_{\{|v|=1\}}|\nabla v|^2\dx=0,\quad\eps\searrow0.
\end{align*}
%\LB{REF for the last thing}
\end{proof}

\begin{proof}[Proof of Proposition \ref{prop:eigenvectors_div}]
We write $p=p_u^0=-\nabla v$ with $v$ as in \eqref{eq:minimal_subgradient} and define $v_\eps$ as in the previous Lemma. Any $q\in K$ can be written as $q=-\nabla w$ with $w\in H^1_0(\Omega,[-1,1])$ and thanks to Lemma~\ref{lem:approx} we have
\begin{align}\label{eq:proof1}
\langle p,p-q\rangle=\int_\Omega\nabla v\cdot\nabla(v-w)\dx=\lim_{\eps\searrow0}\int_\Omega\nabla v_\eps\cdot\nabla(v-w)\dx.
\end{align}
As a first step, we replace $v-w$ by a smooth $z$ with compact support, to obtain with the product rule:
\begin{align*}
\int_\Omega\nabla v_\eps\cdot\nabla z\dx=&\int_\Omega\psi_\eps(v)\nabla v\cdot\nabla z\dx\\
=&\int_\Omega\nabla v\cdot(\nabla[\psi_\eps(v)z]-\psi_\eps'(v)z\nabla v)\dx\\
=&-\int_\Omega\psi_\eps'(v)z|\nabla v|^2\dx
\end{align*}
For the last equality we used that $\psi_\eps(v)z$ is a test function for the minimization of the Dirichlet energy in \eqref{eq:minimal_subgradient} and that the first variation of the Dirichlet integral in direction of the test function vanishes consequently.

Strongly approximating $v-w$ by smooth and compactly supported functions in $H^1_0(\Omega)$ and using above-noted calculation, results in
\begin{align}\label{eq:proof2}
\int_\Omega\nabla v_\eps\cdot\nabla(v-w)\dx=-\int_\Omega\psi_\eps'(v)(v-w)|\nabla v|^2\dx.
\end{align}
Note that
$$\psi_\eps'(t)=\frac{1}{\eps}
\begin{cases}
1,&t\in[-1,-1+\eps]\\
-1,&t\in[1-\eps,1],\\
0,&$else$,
\end{cases}
$$
which we can use to calculate
\begin{align}\label{eq:proof3}
-&\int_\Omega\psi_\eps'(v)(v-w)|\nabla v|^2\dx\notag\\
=&\int_{\{-1\leq v\leq-1+\eps\}}\frac{1}{\eps}(w-v)|\nabla v|^2\dx+\int_{\{1-\eps\leq v\leq1\}}\frac{1}{\eps}(v-w)|\nabla v|^2\dx\notag\\
\geq&-\int_{\{1-\eps\leq|v|\leq 1\}}|\nabla v|^2\dx\notag\\
\to&-\int_{\{|v|=1\}}|\nabla v|^2\dx=0,\quad\eps\searrow0.
\end{align}
Putting \eqref{eq:proof1}--\eqref{eq:proof3} together, we have shown that 
$$\langle p,p-q\rangle=\lim_{\eps\searrow 0}\int_\Omega\nabla v_\eps\cdot\nabla(v-w)\dx=\lim_{\eps\searrow 0}-\int_\Omega\psi_\eps'(v)(v-w)|\nabla v|^2\dx\geq 0,$$
as desired.
\end{proof}

\begin{thm}[Divergence flow]\label{thm:eigenvectors_div-flow}
Let $J$ be given by \eqref{eq:divergence}. Then gradient flow $\partial_tu(t)=-p(t),\;p(t)\in\partial J(u(t))$ with $u(0)=f$ generates a sequence of eigenvectors $p(t)$ for $t>0$.
\end{thm}

\begin{rem}[Geometric structure of eigenvectors]
From \eqref{eq:subdiff} we infer that $p\in\partial J(p)$ if and only if $p=-\nabla v$ where $v\in H^1_0(\Omega,[-1,1])$ and $v=\pm1$ on $(\div u)_\pm=(-\Delta v)_\pm$. Note that $-\Delta v$ is a finite Radon measure in this situation whose support cannot contain an open set since otherwise $v$ would be locally constant there and hence have zero Laplacian. Therefore, eigenvectors are gradients of functions $v$ which are harmonic everywhere on $\Omega\setminus\Gamma$ where $\Gamma:=\supp(-\Delta v)$ is a closed exceptional set, not containing any open set.
\end{rem}
\subsection{Rotation flow}
\label{sec:rot}
Now we fix the dimension $n=2$ and consider the functional
\begin{align}\label{eq:rotation}
\tilde{J}(u):=\int_\Omega|\rot u|:=\sup\left\lbrace\int_\Omega u\cdot \nabla^\perp v\dx\st v\in C_c^\infty(\Omega,[-1,1])\right\rbrace,\;u\in L^2(\Omega,\R^n),
\end{align}
where formally $\rot u=\partial_1u_2-\partial_2u_1$ and $\nabla^\perp=(\partial_2,-\partial_1)^T$. Defining the rotation matrix $R=\begin{pmatrix}
0 & 1 \\ 
-1 & 0
\end{pmatrix}$ which fulfills $\nabla^\perp=R\nabla$ it holds $\tilde{J}(u)=J(Ru)$, where $J$ is given by \eqref{eq:divergence}. As before, $\tilde{J}$ can be expressed by duality as $\tilde{J}=\chi_{\tilde{K}}^*$ where $\tilde{K}:=\{\nabla^\perp v\st v\in H^1_0(\Omega,[-1,1])\}$ and it holds $\tilde{K}=RK$. Due to the invertibility of $R$, the gradient flows with respect to $J$ and $\tilde{J}$ are fully equivalent and the respective solutions are connected via the rotation $R$. In particular, the results from the previous section directly generalize to $\tilde{J}$, meaning that the rotation flow $\partial_tu=-p,\;p\in\partial\tilde{J}(u)$ also generates eigenfunctions $p\in\partial\tilde{J}(p)$.
\begin{thm}[Rotation flow]\label{thm:eigenvectors_rot-flow}
Let $\tilde{J}$ be given by \eqref{eq:rotation}. Then the gradient flow $\partial_tu(t)=-p(t),\;p(t)\in\partial\tilde{J}(u(t))$ with $u(0)=f$ generates a sequence of eigenvectors $p(t)$ for $t>0$.
\end{thm}
\subsection{Divergence-Rotation flow}
Now we define the functional
\begin{align}\label{eq:div-rot}
\calJ(u):=\int_\Omega|\div u|+\int_\Omega|\rot u|= J(u)+\tilde{J}(u)
\end{align}
which measures the sum of the distributional divergence and rotation of a vector field $u\in L^2(\Omega,\R^n)$. Its null-space still contains all gradients of harmonic functions and we denote $\calK:=\partial\calJ(0)$.

\begin{prop}\label{prop:sum_rule}
It holds for all $u\in\dom(\partial\calJ)$
\begin{align}
\partial\calJ(u)=\partial J(u)+\partial \tilde{J}(u).
\end{align}
Furthermore, the sets $\partial J(u)$ and $\partial\tilde{J}(u)$ are orthogonal.
\end{prop}
Crucial for the proof of this sum rule is the \emph{Helmholtz decomposition}, which can be phrased as follows.
\begin{thm}[Helmholtz decomposition]
Let $\Omega\subset\R^n$ be an arbitrary domain and let
\begin{align*}
L^2_\sigma(\Omega,\R^n)&:=\overline{\{\varphi\in C_c^\infty(\Omega,\R^n)\st\div\varphi=0\}}^{\norm{\cdot}_{L^2}},\\
G^2(\Omega,\R^n)&:=\{-\nabla v\in L^2(\Omega,\R^n)\st v\in L^2_\mathrm{loc}(\Omega)\}.
\end{align*}
Then the \emph{Helmholtz decomposition} 
$$L^2(\Omega,\R^n)=L^2_\sigma(\Omega,\R^n)\oplus G^2(\Omega,\R^n)$$
holds and is orthogonal. The orthogonal projection onto the closed subspace $L^2_\sigma(\Omega,\R^n)$ of $L^2(\Omega,\R^n)$ is called \emph{Helmholtz projection} and denoted by $\Pi_\sigma$.
\end{thm}
\begin{proof}
The proof can be found in \cite[Lemma~2.5.1, 2.5.2]{sohr2012navier}.
\end{proof}

\begin{proof}[Proof of Proposition~\ref{prop:sum_rule}]
Let $u\in\dom(\partial\calJ)$ and $p\in\partial\calJ(u)$. In particular, $p\in\calK$ and hence it holds for arbitrary $v\in L^2(\Omega,\R^n)$ that 
$$\langle p,v\rangle\leq\calJ(v)=J(v)+\tilde{J}(v).$$
In particular, one has
\begin{align*}
&\langle (1-\Pi_\sigma)p,v\rangle=\langle p,(1-\Pi_\sigma)v\rangle\leq J((1-\Pi_\sigma)v)=J(v),\\
&\langle\Pi_\sigma p,v\rangle=\langle p,\Pi_\sigma v\rangle\leq\tilde{J}(\Pi_\sigma v)=\tilde{J}(v),
\end{align*}
where we used 3. in Proposition~\ref{prop:properties_one-hom} and the self-adjointness of projections. This shows $(1-\Pi_\sigma)p\in K$ and $\Pi_\sigma p\in\tilde{K}$ and hence
\begin{align*}
\langle(1-\Pi_\sigma)p,u\rangle&\leq J(u),\\
\langle\Pi_\sigma p,u\rangle&\leq\tilde{J}(u).
\end{align*}
If we assumed that one of the inequalities was strict, then using $\langle p,u\rangle=\calJ(u)$ implied
\begin{align*}
J(u)+\tilde{J}(u)=\calJ(u)=\langle p,u\rangle=\langle(1-\Pi_\sigma)p,u\rangle+\langle\Pi_\sigma p,u\rangle<J(u)+\tilde{J}(u)
\end{align*}
which is a contradiction. Therefore, we have
\begin{align*}
\langle(1-\Pi_\sigma)p,u\rangle&=J(u),\\
\langle\Pi_\sigma p,u\rangle&=\tilde{J}(u),
\end{align*}
which lets us conclude that $(1-\Pi_\sigma)p\in\partial J(u)$ and $\Pi_\sigma p\in\partial\tilde{J}(u)$. Hence, we have established $p\in\partial J(u)+\partial\tilde{J}(u)$ which concludes the first part of the proof. Orthogonality of $\partial J(u)$ and $\partial\tilde{J}(u)$ follows from the fact that any subgradient $p_J$ and $p_{\tilde{J}}$ in these sets can be written as $p_J=-\nabla v$ and $p_{\tilde{J}}=\nabla^\perp w$ with $v,w\in H^1_0(\Omega,[-1,1])$ respectively. Approximating $w$ strongly by compactly supported and smooth functions, it follows from an integration by parts using $\div\nabla^\perp w=0$ and the zero boundary conditions of $v$ that
$$\langle p_j,p_{\tilde{J}}\rangle=-\langle\nabla v,\nabla^\perp w\rangle=\langle v,\div\nabla^\perp w\rangle=0.$$
\end{proof}

\begin{cor}\label{cor:min_subgrad_div-rot}
Let $u\in\dom(\partial\calJ)$. Then the subgradient of minimal norm in $\partial\calJ(u)$ is given by the sum of the subgradients of minimal norm in $\partial J(u)$ and $\partial\tilde{J}(u)$, respectively.
\end{cor}
\begin{proof}
The proof follows directly from Proposition~\ref{prop:sum_rule}.
\end{proof}

\begin{lemma}\label{lem:prop_subgrad}
Let $p_J$ and $p_{\tilde{J}}$ be eigenvectors with eigenvalue 1 of $\partial J$ and $\partial\tilde{J}$, respectively. Then it holds
\begin{enumerate}
\item $\langle p_J,p_{\tilde{J}}\rangle=0$,
\item $p_J\in\calN(\tilde{J})$, $p_{\tilde{J}}\in\calN(J)$,
\item $p:=p_j+p_{\tilde{J}}$ is an eigenvector with eigenvalue 1 of $\partial\calJ$
\end{enumerate}
\end{lemma}
\begin{proof}
Ad 1.: According to Proposition~\ref{prop:sum_rule} this holds for general subgradients.

Ad 2.: We only proof $p_J\in\calN(\tilde{J})$, the proof of the second inclusion working analogously. We calculate
$$\tilde{J}(p_J)=J(Rp_J)=\sup_{p\in K}\langle p,Rp_J\rangle=\sup_{p\in K}\langle R^Tp,p_J\rangle.$$
Due to the definition of $K$ we can write $p_J=-\nabla v$ and $p=-\nabla w$ with $v,w\in H^1_0(\Omega,[-1,1])$. This yields
$$\langle R^Tp,p_J\rangle=\langle R\nabla w,-\nabla v\rangle=-\langle\nabla^\perp w,\nabla v\rangle.$$
By the same density argument as above we obtain $\langle R^Tp,p_J\rangle=0$ and hence $\tilde{J}(p_J)=0$.

Ad 3.: Using that $p_J$ and $p_{\tilde{J}}$ are eigenvectors, it holds
\begin{align*}
\langle p,p\rangle&=\langle p_J,p_J\rangle+\langle p_{\tilde{J}},p_{\tilde{J}}\rangle+2\overbrace{\langle p_J,p_{\tilde{J}}\rangle}^{=0}\\
&=J(p_J)+\tilde{J}(p_{\tilde{J}})\\
&=J(p)+\tilde{J}(p)\\
&=\calJ(p).
\end{align*}
For the third equality we used 2. together with 3. in Proposition~\ref{prop:properties_one-hom}.  
\end{proof}

From here on we easily infer 

\begin{thm}[Div-Rot flow]\label{thm:eigenvectors_div-rot-flow}
Let $\calJ$ be given by \eqref{eq:div-rot}. Then the gradient flow $\partial_t u(t)=-p(t),\;p\in\partial\calJ(u(t))$ with $u(0)=f$ yields a sequence of eigenvectors $p(t)$ for $t>0$. Furthermore, $(1-\Pi_\sigma)p(t)$ and $\Pi_\sigma p(t)$ are eigenvectors of $\partial J$ and $\partial\tilde{J}$, respectively.
\end{thm}
\begin{proof}
The gradient flow selects subgradients of minimal norm, all of which are given by the sum of the subgradients of minimal norm in $\partial J(u)$ and $\partial\tilde{J}(u)$ according to Corollary~\ref{cor:min_subgrad_div-rot}. Due to Thms~\ref{thm:eigenvectors_div-flow} and \ref{thm:eigenvectors_rot-flow}, each subgradient of minimal norm is an eigenvector of $\partial J$ and $\partial\tilde{J}$, respectively, and by 3. in Lemma~\ref{lem:prop_subgrad} we conclude that their sum is an eigenvector of $\partial\calJ$. The uniqueness of the Helmholtz decomposition concludes the proof.
\end{proof}

\section{Finite extinction time and extinction profiles}\label{sec:extinction}
Let us now consider a general gradient flow $\partial_tu(t)=-p(t),\;p(t)\in\partial J(u(t))$, with arbitrary $J\in\calC$ and initial condition $u(0)=f$\change{, where the class $\calC$ is defined in Definition~\ref{def:class_C}.} It is well-known (cf.~\cite{benning2013ground,bungert2018solution}, for instance) that if the datum $f$ is an eigenvector fulfilling $\lambda f\in\partial J(f)$, the corresponding solution of the gradient flow is given by $u(t)=\max(1-\lambda t,0)f$. Hence, there exists a time $T:=1/\lambda$, referred to as extinction time, such that $u(t)=0=\bar{f}$ for all $t\geq T$. In general, Theorem~\ref{thm:decrease} tells us that $u(t)$ only converges to $\bar{f}$ strongly as $t\to\infty$. In the following, we will investigate under which conditions on the data $f$ and the functional $J\in\calC$ this limit is attained in finite time, i.e., 
\begin{align*}
T^*(f):=\inf\{t>0\st u(t)=\bar{f}\}<\infty.
\end{align*}
Furthermore, we derive lower and upper bounds of the extinction time which will depend on the (dual) norm of the data $f$. Scale invariant upper bounds in the special case of total variation flow and related equations were shown in \cite{giga2011scale,giga2015fourth}.
\subsection{Finite extinction time}
First, we give a statement that can be interpreted as conservation of mass under the gradient flow.
\begin{lemma}[Conservation of mass]\label{lem:conserv_mean}
Let $u$ solve \eqref{eq:gradient_flow}. Then it holds $\overline{u(t)}=\bar{f}$ for all $t>0$.
\end{lemma}
\begin{proof}
Since $p(t)$ lies in $\calN(J)^\perp$ for all $t>0$, the same holds for $u(t)-f=-\int_0^tp(s)\d s$. Therefore, $0=\overline{u(t)-f}=\overline{u(t)}-\bar{f}$, by the linearity of the projection. 
\end{proof}

\begin{thm}[Upper bound of extinction time]\label{thm:finite_ext_time}
Let $u$ solve \eqref{eq:gradient_flow}. If there is $C>0$ such that 
\begin{align}\label{ineq:poincare}
\norm{u(t)-\bar{f}}\leq CJ(u(t)),\quad\forall t>0,
\end{align}
then it holds
\begin{align}\label{ineq:upper_bound_ext}
T^*(f)\leq C\norm{f-\bar{f}}
\end{align}
\end{thm}
\begin{proof}
The proof is analogous to the finite dimensional case, treated in \cite{burger2016spectral}. Using \eqref{ineq:poincare}, it follows
\begin{align*}
&\frac{1}{2}\frac{\d}{\d t}\norm{u(t)-\bar{f}}^2=\langle u(t)-\bar{f},\partial_t u(t)\rangle=-\langle u(t)-\bar{f},p(t)\rangle\\
=&-\langle u(t),p(t)\rangle=-J(u(t))\leq-\frac{1}{C}\norm{u(t)-\bar{f}}.
\end{align*}
This readily implies 
$$\frac{\d}{\d t}\norm{u(t)-\bar{f}}\leq-\frac{1}{C}$$
and, integrating this equation,
$$\norm{u(t)-\bar{f}}\leq\norm{f-\bar{f}}-\frac{t}{C}.$$
Therefore, for $t\geq\norm{f-\bar{f}}C$ it holds $u(t)=\bar{f}$ which concludes the proof. 
\end{proof}

\begin{rem}
By Lemma~\ref{lem:conserv_mean}, a sufficient condition for \eqref{ineq:poincare} to hold is the validity of the Poincar\'{e}-type inequality
\begin{align}\label{ineq:general_poincare}
\norm{u-\bar{u}}\leq CJ(u),\quad\forall u\in\H.
\end{align}
\end{rem}

\begin{rem}
Revisiting the example of a datum fulfilling $\lambda f\in\partial J(f)$, we observe that the upper bound \eqref{ineq:upper_bound_ext} is sharp. It holds $u(t)=\max(1-\lambda t,0)f$ and consequently we have 
$$\frac{\norm{u(t)-\bar{f}}}{J(u(t))}=\frac{\norm{u(t)}}{J(u(t))}=\frac{\norm{f}}{J(f)}$$
and, hence, the smallest possible constant in \eqref{ineq:poincare} is given by $C=\frac{\norm{f}}{J(f)}$. This implies that 
$$C\norm{f-\bar{f}}=C\norm{f}=\frac{\norm{f}^2}{J(f)}=\frac{1}{\lambda}=T^*(f).$$
\end{rem}

\begin{example}
Since any $J\in\calC$ is a norm on the subspace $\mathcal{V}=\dom(J)\cap\calN(J)^\perp\subset\H$, the Poincar\'{e} inequality \eqref{ineq:general_poincare} always holds if $\H$ is finite dimensional.
\end{example}

\begin{example}
Since for $n=1,2$ and $u\in\bv(\Omega)$ where $\Omega\subset\R^n$ is a bounded domain one has the Poincar\'{e} inequality \cite{ambrosio2000functions}
$$\norm{u-\bar{u}}\leq C\,\tv(u),$$
where $\bar{u}:=|\Omega|^{-1}\int_\Omega u\dx$ and $C>0$, the one and two-dimensional $\tv$-flow $\partial_tu(t)=-p(t),\;p(t)\in\partial \tv(u(t))$ becomes extinct in finite time for all initial conditions $u(0)=f\in L^2(\Omega)$. 
\end{example}

\begin{example}
Since not even on bounded domains there is an embedding from $L^1(\Omega)$ to $L^2(\Omega)$, one cannot expect finite extinction time for the $L^1$-flow introduced in Section~\ref{sec:l1flow}. Indeed one has 
$$u(t)=f-\int_0^tp(s)\d s$$
where $\norm{p(s)}_{L^\infty}\leq 1$. This yields $u(t)\geq f-t$ almost everywhere in $\Omega$ for all $t\geq 0$. Hence, if $f$ is an unbounded $L^2$-function, the extinction time is infinite. Hence, the quite strong property of a flow yielding a sequence of non-linear eigenvectors still does not imply a finite extinction time. 
\end{example}

We can also give a converse statement of Theorem~\ref{thm:finite_ext_time}, namely that the existence of a finite extinction time implies the validity of a Poincar\'{e}-type inequality.

\begin{prop}
Let $u$ solve \eqref{eq:gradient_flow} and assume that there exists $T>0$ such that the solutions of the gradient flow satisfies $u(t)=\bar{f}$ for all $t\geq T$. Then, one has the estimate
\begin{align}
\norm{u(t)-\bar{f}}\leq  \sqrt{T-t}\sqrt{J(u(t))},\quad\forall 0\leq t\leq T.
\end{align}
In particular, for $t=0$ and initial data satisfying $\norm{f-\bar{f}}=1$ it this implies
$$ \norm{f-\bar{f}}\leq T J(f).$$
\end{prop}
\begin{proof}
It holds $u(t)-\bar{f}=\int_t^T p(s)\d s$ which implies with the H\"{o}lder inequality and 5. in Theorem~\ref{thm:brezis2}
\begin{align*}
\norm{u(t)-\bar{f}}&\leq\int_t^T\norm{p(s)}\d s\\
&\leq\sqrt{T-t}\sqrt{\int_t^T\norm{p(s)}^2\d s}\\
&=\sqrt{T-t}\sqrt{\int_T^t\frac{\d}{\d s}J(u(s))\d s}\\
&=\sqrt{T-t}\sqrt{J(u(t))}.
\end{align*}
Noticing $\norm{f-\bar{f}}=1=\norm{f-\bar{f}}^2$ we obtain the assertion for $t=0$.
 %\LB{Does this help?} Furthermore, it holds
%$$\frac{\d}{\d t}\norm{u(t)-\bar{f}}=-\frac{J(u(t)}{\norm{u(t)-\bar{f}}}.$$
%Integrating from $t$ to $T$ and using $u(T)=\bar{f}$ yields
%$$\norm{u(t)-\bar{f}}=\int_t^T\frac{J(u(s))}{\norm{u(s)-\bar{f}}}\d s.$$
%Hence
%$$\frac{(T-s)J(u(s))}{\norm{u(s)-\bar{f}}}\to 0$$
%as $s\to T$. 
\end{proof}

While upper bounds of the extinction time like \eqref{ineq:upper_bound_ext} which depend on some Poincar\'{e} constant have already been considered in the literature \cite{andreu2002some,hauer2017kurdyka}, we can also give a sharp lower bound of the extinction time which is novel to the best of our knowledge.

\begin{prop}[Lower bound of extinction time]
It holds
\begin{align}\label{ineq:lower_bound_ext}
T^*(f)\geq \norm{f}_*,
\end{align}
where $\norm{f}_*:=\sup_{p\in\calN(J)^\perp}\frac{\langle f,p\rangle}{J(p)}$ denotes the dual norm of $f$ with respect to $J$.
\end{prop} 
\begin{proof}
Using $f-\bar{f}=\int_0^{T^*(f)}p(t)\d t$ it holds
\begin{align*}
\norm{f}_*=\sup_{p\in\calN(J)^\perp}\frac{\langle f,p\rangle}{J(p)}=\sup_{p\in\calN(J)^\perp}\frac{1}{J(p)}\int_0^{T^*(f)}{\langle p(t),p\rangle}\d t\leq T^*(f),
\end{align*}
\change{where we used that due to \eqref{eq:subdiff_0} it holds $\langle p(t),p\rangle\leq J(p)$ for all $t>0$.}
\end{proof}
\begin{rem}
Since according to \cite{bungert2018solution} the variational problem \eqref{eq:var_prob} has the minimal extinction time $\norm{f}_*$, we conclude that the extinction time of the gradient flow is always larger or equal the extinction time of the variational problem. Furthermore, this lower bound is sharp since in the case of a spectral decomposition the gradient flow and the variational problem have the same primal solution according to Theorem~\ref{thm:equivalence_gf_vp}.
\end{rem}

\subsection{Extinction profiles: the general case}
Let us now assume that the gradient flow becomes extinct in finite time and to simplify notation we consider the case of a datum $f$ fulfilling $\bar{f}=0$ without loss of generality (cf.~Lemma~\ref{lem:equivalence_init_cond}). In \cite{andreu2002some} it was shown for the special case of 2D total variation flow with initial condition $f\in L^\infty(\Omega)$ that $T^*(f)<\infty$ and there is an increasing sequence of times $t_n$ such that $t_n\to T^*(f)$ and
\begin{align}\label{eq:ext_conv}
\lim_{n\to\infty}w(t_n)= p^*
\end{align}
strongly in $L^2(\Omega)$ where 
\begin{align}
w(t):=\begin{cases}
\frac{u(t)}{T^*(f)-t},&0<t<T^*(f),\\
0,&$else$.
\end{cases}
\end{align}
Here $p^*\neq 0$ is an eigenvector and is referred to as \emph{extinction profile}. Note that $w(t)$ approximates the negative of the left-derivative of $t\mapsto u(t)$ at $t=T^*(f)$ which---opposed to the right-derivative---is not guaranteed to exist. The proof in \cite{andreu2002some} is highly technical and uses a lot of structure that comes from the explicit form of the functional $J=TV$. However, in our general framework this result can be harvested very easily and for general functionals $J\in\calC$. First, we show that, if the strong limit \eqref{eq:ext_conv} exists, it is an eigenvector.

\begin{thm}\label{thm:ext_profile}
Let $u(t)$ solve the gradient flow \eqref{eq:gradient_flow} with extinction time $T^*(f)<\infty$, and assume that there is an increasing sequence $t_n\to T^*(f)$ such that $w(t_n)\to p^*$ as $n\to\infty$ holds for some $p^*\in\H$. Then $p^*\in\partial J(p^*)$ and if the Poincar\'{e} inequality \eqref{ineq:poincare} holds one has $p^*\neq 0$. 
\end{thm}
\begin{proof}
Let us show $p^*\in K$, first. Due to $u(T^*(f))=0$ and $p(t)=-\partial_t u(t)$ for $t>0$ it holds 
$$u(t_n)=\int_{t_n}^{T^*(f)} p(t) \d t$$
and, hence, we calculate for arbitrary $v\in\H$
$$\left\langle\frac{u(t_n)}{T^*(f)-t_n},v\right\rangle=\frac{1}{T^*(f)-t_n}\int_{t_n}^{T^*(f)}\langle p(t),v\rangle\d t\leq J(v),\quad\forall n\in\N,$$
where we used that $p(t)\in K$ for all $t>0$. By the closedness of $K$ we infer that also $p^*\in K$ holds. 

To show that $p^*\in\partial J(p^*)$, we calculate using lower semi-continuity of~$J$:
\begin{align*}
J(p^*)&\leq\liminf_{n\to\infty}\frac{1}{T^*(f)-t_n}J(u(t_n))\\
&=\liminf_{n\to\infty}\frac{1}{T^*(f)-t_n}\langle p(t_n),u(t_n)\rangle\\
&=\liminf_{n\to\infty}\langle p(t_n),w(t_n)\rangle.
\end{align*}
Now we claim (cf.~Lemma~\ref{lem:techn_est} below) that $\lim_{t\nearrow T^*(f)}\norm{p(t)-w(t)}=0$ which, together with the strong convergence of $w(t_n)$ to $p^*$, immediately implies
\begin{align*}
J(p^*)&\leq\liminf_{n\to\infty}\langle p(t_n),w(t_n)\rangle\\
&=\liminf_{n\to\infty}\langle p(t_n)-w(t_n),w(t_n)\rangle+\norm{w(t_n)}^2\\
&=\lim_{n\to\infty}\norm{w(t_n)}^2=\norm{p^*}^2,
\end{align*}
and, hence, $p^*$ is an eigenvector. Here we used that $\langle p(t_n)-w(t_n),w(t_n)\rangle\leq\norm{p(t_n)-w(t_n)}\norm{w(t_n)}$ which converges to zero since $\norm{w(t_n)}$ is uniformly bounded in $n$. 

To show that $p^*\neq 0$ in case of \eqref{ineq:poincare}, we observe that
\begin{align*}
\frac{\d}{\d t}\frac{1}{2}\norm{u(t)}^2=\langle\partial_t u(t),u(t)\rangle=-\langle p(t),u(t)\rangle=-J(u(t))
\end{align*}
holds, which implies with \eqref{ineq:poincare} that
\begin{align*}
\frac{\d}{\d t}\norm{u(t)}=\frac{-J(u(t))}{\norm{u(t)}}\leq-\frac{1}{C}.
\end{align*}
Integrating this inequality from $0\leq t< T^*(f)$ to $T^*(f)$ yields
$$-\norm{u(t)}\leq-\frac{1}{C}(T^*(f)-t)$$
and, hence, 
$$\norm{w(t)}=\frac{\norm{u(t)}}{T^*(f)-t}\geq\frac{1}{C}>0,\quad\forall 0\leq t<T^*(f).$$
This implies that $p^*\neq 0$ holds.
%\LB{Special case spectral decomposition, not necessary any more:} It remains to be shown that $p^*\in\partial J(p^*)$. To this end, we compute, using the lower semi-continuity of $J$ and Corollary~\ref{rem:triangle_integral},
%\begin{align*}
%J(p^*)&\leq\liminf \frac{1}{h} J\left(\int_{T-h}^T p(t)\d t\right)\leq\liminf\frac{1}{h}\int_{T-h}^TJ(p(t))\d t\\
%&\leq\liminf\frac{1}{h}\int_{T-h}^T\norm{p(t)}^2\d t.
%\end{align*}
%Since the function $t\mapsto\norm{p(t)}$ is non-increasing and bounded from below, the left-limit $L:=\lim_{t\nearrow T}\norm{p(t)}^2$ exists and due to the it holds 
%$$\frac{1}{h}\int_{T-h}^T\norm{p(t)}^2\d t\to L$$
%as $h\searrow 0$. This shows $J(p^*)\leq L$. On the other hand, we have
%\begin{align*}
%\norm{p^*}^2=&\lim\frac{1}{h^2}\norm{\int_{T-h}^T p(t)\d t}^2\\
%=&\lim\frac{1}{h^2}\int_{T-h}^T\int_{T-h}^T\langle p(s),p(t)\rangle\d s\d t\\
%=&\lim\frac{1}{h^2}\int_{T-h}^T\left[\int_{T-h}^t\langle p(s),p(t)\rangle\d s+\int_{t}^T\langle p(s),p(t)\rangle\d s\right]\d t\\
%=&\lim\frac{1}{h^2}\int_{T-h}^T\left[\int_{T-h}^t\norm{p(t)}^2\d s+\int_{t}^T\norm{p(s)}^2\d s\right]\d t,
%\end{align*}
%where we used the orthogonality of the increments (cf.~Theorem~\ref{thm:orthogonality}). An easy calculation shows that this limit is given by $L$ which shows $J(p^*)\leq L=\norm{p^*}^2$, as desired.
\end{proof}

\begin{lemma}\label{lem:techn_est}
Under the conditions of Theorem~\ref{thm:ext_profile} it holds $\norm{p(t)-w(t)}\to 0$ as $t\nearrow T^*(f)$.
\end{lemma}
\begin{proof}
Using the relation $\partial_t w(t)=(\partial_t u(t)+w(t))/(T^*(f)-t)$, the chain rule, and $p(t)\in\partial J(u(t))$, we obtain the estimate
\begin{align}\label{ineq:en_diss}
\frac{\d}{\d t}&\left(J(w(t))-\frac{1}{2}\norm{w(t)}^2\right)\notag
=\frac{\d}{\d t}\left(\frac{J(u(t))}{T^*(f)-t}-\frac{1}{2}\norm{w(t)}^2\right)\notag\\
&=\frac{\langle p(t),\partial_t u(t)\rangle}{T^*(f)-t}+\frac{J(u(t))}{(T^*(f)-t)^2}-\langle w(t),\partial_t w(t)\rangle\notag\\
&=\frac{\langle p(t),\partial_t u(t)\rangle}{T^*(f)-t}+\frac{\langle p(t),u(t)\rangle}{(T^*(f)-t)^2}-\langle w(t),\partial_t w(t)\rangle\notag\\
&=\langle p(t),\partial_t w(t)\rangle-\langle w(t),\partial_t w(t)\rangle\notag
=-\langle \partial_t u(t)+w(t),\partial_t w(t)\rangle\notag\\
&=-(T^*(f)-t)\norm{\partial_t w(t)}^2=-\frac{\norm{p(t)-w(t)}^2}{T^*(f)-t}\leq 0
\end{align}
Since $u(t)\in\dom(J)$ for all $t>0$, we can without loss of generality assume that $u(0)=f\in\dom(J)$, as well. This implies
$$J(w(t))-\frac{1}{2}\norm{w(t)}^2\leq J(w(0))-\frac{1}{2}\norm{w(0)}^2\leq C,\;0\leq t<T^*(f),$$
for a constant $C>0$ since $w(0)=u(0)/T^*(f)=f/T^*(f)$. Integrating \eqref{ineq:en_diss} yields
\begin{align*}
\int_s^t\frac{\norm{p(\tau)-w(\tau)}^2}{T^*(f)-\tau}\d\tau&=J(w(s))-\frac{1}{2}\norm{w(s)}^2-J(w(t))+\frac{1}{2}\norm{w(t)}^2\\
&\leq C+\frac{1}{2}\norm{w(t)}^2
\leq C',\quad 0\leq s<t<T^*(f),
\end{align*}
for a constant $C'>0$. Letting $t$ tend to $T^*(f)$, this implies that $t\mapsto\frac{\norm{p(t)-w(t)}^2}{T^*(f)-t}\in L^1(0,T^*(f))$. Hence, necessarily $\lim_{t\nearrow T^*(f)}\norm{p(t)-w(t)}=0$ has to hold.
\end{proof}

The weak limit in \eqref{eq:ext_conv} always exists, as the following proposition states.

\begin{prop}\label{prop:exist_weak_limit}
There exists a  increasing sequence $t_n\to T^*(f)$ and $p^*\in\H$ such that 
$$w(t_n)\rightharpoonup p^*,\;n\to\infty,$$
weakly in $\H$.
\end{prop}
\begin{proof}
From $u(t)=\int_{t}^{T^*(f)}p(s)\d s$ it follows using 2. in Theorem~\ref{thm:brezis2} and that $s\mapsto\norm{p(s)}$ is non-increasing 
\begin{align*}
\norm{w(t)}=\norm{\frac{u(t)}{T^*(f)-t}}\leq\frac{1}{T^*(f)-t}\int_{t}^{T^*(f)}\norm{p(s)}\d s\leq\norm{p(t)}\leq \frac{2}{T^*(f)}\norm{f},
\end{align*}
which holds for all $T^*(f)/2<t<T^*(f)$. Hence, by the weak compactness of the closed unit ball in Hilbert spaces, there is a increasing sequence $(t_n)$ converging to $T^*(f)$ and some $p^*\in\H$ such that $w(t_n)\rightharpoonup p^*$ weakly in $\H$. 
\end{proof}

To ensure the existence of the strong limit one needs additional regularity, as the following corollary states.

\begin{cor}\label{cor:strong_limit}
Assume that there is some compactly embedded space $\X\Subset\H$ and $c>0$ such that $\sup_{n\in\N}\norm{w(t_n)}_\X\leq c$. Then the convergence of $w(t_n)$ to $p^*$ is strong in~$\H$ and Theorem~\ref{thm:ext_profile} is applicable.
\end{cor}

Using our abstract framework, we can obtain the results in \cite{andreu2002some} for more-dimensional total variation flow very straightforwardly.
\begin{example}[Total variation flow in arbitrary dimension]
We consider the $\tv$-flow in dimension $n\geq 3$ with bounded domain $\Omega\subset\R^n$. In this case, there is no general Poincar\'{e} inequality available and, hence, no finite extinction time can be expected for arbitrary initial datum $f\in L^2(\Omega)$. However, for $f\in L^\infty(\Omega)$ one can show a finite extinction time. It is well-known that in this case the essential supremum of the solution $u(t)$ of the $\tv$-flow remains uniformly bounded in $t$. Hence, we show that the Poincar\'{e} inequality 
\begin{align*}
\norm{u-\bar{u}}_{L^2(\Omega)}\leq C\tv(u)
\end{align*}
holds for functions $u\in\{u\in \bv(\Omega)\cap L^\infty(\Omega)\st\norm{u}_{L^\infty(\Omega)}\leq c\}$. Assume there is a sequence of functions $u_k\in\bv(\Omega)\cap L^\infty(\Omega)$ with zero mean such that
\begin{align*}
\norm{u_k}_{L^2(\Omega)}>k\tv(u_k),\quad\norm{u_k}_{L^\infty(\Omega)}\leq c,\quad\forall k\in\N.
\end{align*}
Since $\norm{u_k}_{L^2(\Omega)}\leq\sqrt{|\Omega|}\norm{u_k}_{L^\infty(\Omega)}\leq\sqrt{|\Omega|}c<\infty$, we can set $\norm{u_k}_{L^2(\Omega)}=1$. Passing to a subsequence, we can furthermore assume that $u_k$ converges strongly to some $u$ in $L^1(\Omega)$. Since $\tv(u)\leq\liminf_{k\to\infty}\tv(u_k)=0$, by the lower semi-continuity of the total variation, we infer that $u$ is in $\bv(\Omega)$ and is constant almost everywhere in $\Omega$. Having zero mean, $u\equiv 0$ has to hold. Furthermore, from $0\leq\norm{u_k}_{L^2(\Omega)}\leq\sqrt{c\norm{u_k}_{L^1(\Omega)}}\to 0$ as $k\to\infty$ we infer that $u_k\to 0$ in $L^2(\Omega)$ which is a contradiction to $\norm{u_k}_{L^2(\Omega)}=1$ for all $k\in\N$. Now Theorem~\ref{thm:decrease} shows the existence of a finite extinction time.

To show that the strong limit $\lim_{n\to\infty}w(t_n)$ exists on some increasing sequence $t_n\to T^*(f)$ and is a non-trivial eigenvector of $\tv$, it suffices to observe that the space $\X:=L^\infty(\Omega)\cap\bv(\Omega)$ is compactly embedded in $\H:=L^2(\Omega)$ and use \cite[Eq. (5.4)]{andreu2002some} in order to apply Corollary~\ref{cor:strong_limit} and Theorem~\ref{thm:ext_profile}.
\end{example}
\change{
\begin{rem}[Sharp convergence rate]
We would like to remark that the results of this section imply that the solution of the gradient flow $u(t)$ converges with rate $T^*(f)-t$ and if the Poincar\'{e} inequality \eqref{ineq:poincare} holds this rate is sharp. The upper rate is established through the boundedness of $\norm{w(t)}$ as shown in the proof of Proposition~\ref{prop:exist_weak_limit}. The sharpness follows from the lower bound on $\norm{w(t)}$ as established in the end of the proof of Theorem~\ref{thm:ext_profile}.
\end{rem}
}

\subsection{Extinction profiles: the spectral decomposition case}

Now we again specialize on the case that the gradient flow generates a sequence of eigenvectors and assume that it becomes extinct at time $T^*(f)<\infty$. In this case, we prove that extinction profiles are always eigenvectors without any compactness assumptions. In addition, since the gradient flow is equivalent to the variational problem, it holds $T^*(f)=\norm{f}_*$ according to \cite{bungert2018solution}. However, we can give another formula in terms of the extinction profile which allows us to classify all possible extinction profiles as maximizers of a certain functional.

\begin{thm}\label{thm:formula_ext_time}
Let the gradient flow \eqref{eq:gradient_flow} generate a sequence of eigenvectors $p(t)$ for $t>0$ and have the minimal extinction time $0<T^*(f)<\infty$. Furthermore, let $p^*\in\H$ be such that $w(t_n)\rightharpoonup p^*$, where $(t_n)$ is an increasing sequence converging to $T^*(f)$ as $n\to\infty$ (cf.~Proposition~\ref{prop:exist_weak_limit}). Then the following statements are true:
\begin{enumerate}
\item If \eqref{ineq:poincare} holds, then $p^*\neq0$,
\item $p(t)\in\partial J(p^*)$ for all $0<t<T^*(f)$,
\item $p^*\in\partial J(p^*)$,
\item $T^*(f)=\frac{\langle f,p^*\rangle}{J(p^*)}=\sup_{p\in\calN(J)^\perp}\frac{\langle f,p\rangle}{J(p)}$.
\end{enumerate}
\end{thm}
\begin{proof}
For a concise notation we abbreviate $T:=T^*(f)$.
Ad 1.:
We have already seen in the proof of Theorem~\ref{thm:ext_profile} that the Poincar\'{e} inequality \eqref{ineq:poincare} implies that $\norm{w(t)}\geq \frac{1}{C}>0$ for all $0\leq t<T^*(f)$. Since, however, $w(t)$ converges only weakly to $p^*$ this is not yet sufficient to prove $p^*\neq0$. Instead we consider the scalar product $\langle w(t),f\rangle$ and use the identities
\begin{align*}
w(t)&=\frac{1}{T-t}\int_t^{T}p(s)\d s,\\
f&=\int_0^{T}p(r)\d r
\end{align*}
to infer
\begin{align}
\langle w(t),f\rangle&=\frac{1}{T-t}\int_0^T\int_t^T\langle p(s),p(r)\rangle\d s\d r\notag\\
&=\frac{1}{T-t}\left[\int_0^t\int_t^T\langle p(s),p(r)\rangle\d s\d r+\int_t^T\int_t^T\langle p(s),p(r)\rangle\d s\d r\right]\notag\\\label{ineq:1}
&=\frac{1}{T-t}\left[\int_0^t\int_t^T\norm{p(s)}^2\d s\d r+\int_t^T\int_t^T\langle p(s),p(r)\rangle\d s\d r\right],
\end{align}
where we used the orthogonality of increments (cf.~Theorem~\ref{thm:orthogonality}). Now we observe that by the Cauchy-Schwarz inequality
$$\norm{w(t)}\leq\frac{1}{T-t}\int_t^T\norm{p(s)}\d s\leq\frac{1}{\sqrt{T-t}}\left(\int_t^T\norm{p(s)}^2\d s\right)^\frac{1}{2},$$
and, hence, 
\begin{align}\label{ineq:2}
\frac{1}{T-t}\int_t^T\norm{p(s)}^2\d s\geq\norm{w(t)}^2\geq \frac{1}{C^2}.
\end{align}
Furthermore, we calculate the following integral by splitting the square integration domain $[t,T]^2$ into integration over two triangles
\begin{align}
\int_t^T\int_t^T\langle p(s),p(r)\rangle\d s\d r&=\int_t^T\int_t^s\langle p(s),p(r)\rangle\d r\d s+\int_t^T\int_t^r\langle p(s),p(r)\rangle\d s\d r\notag\\
&=\int_t^T(s-t)\norm{p(s)}^2\d s+\int_t^T(r-t)\norm{p(r)}^2\d r\notag\\\label{ineq:3}
&\geq 0,
\end{align}
where we used Theorem~\ref{thm:orthogonality} again. Combining \eqref{ineq:1}, \eqref{ineq:2}, and \eqref{ineq:3}, we infer
$$\langle w(t),f\rangle\geq \frac{t}{C^2},\quad0\leq t<T,$$
and, thus, $w(t_n)$ cannot converge weakly to zero for $n\to\infty$.
\\
Ad 2.: According to Theorem~\ref{thm:equivalence_gf_vp} it holds $p(t)\in\partial J(u(s))$ for $t\leq s$. This implies
\begin{align*}
J(p^*)\leq&\liminf_{n\to\infty}J(w(t_n))=\liminf_{n\to\infty}\frac{1}{T^*(f)-t_n}J(u(t_n))=\liminf_{n\to\infty}\frac{1}{T^*(f)-t_n}\langle p(t),u(t_n)\rangle\\
=&\liminf_{n\to\infty}\langle p(t),w(t_n)\rangle=\langle p(t),p^*\rangle,\quad\forall 0\leq t<T^*(f),
\end{align*}
which shows $p(t)\in\partial J(p^*)$.
\\
Ad 3.: The equality $J(p^*)=\langle p(t),p^*\rangle$ implies
$$J(p^*)=\frac{1}{T^*(f)-s}\int_s^{T^*(f)}\langle p(t),p^*\rangle\d t=\langle w(s),p^*\rangle,\quad\forall 0\leq s<T.$$
Using this with $s=t_n$ and taking the limit $n\to\infty$ yields $J(p^*)=\norm{p^*}^2$ due to the weak convergence of $w(t_n)$ to $p^*$. 
\\
Ad 4.: We use 2. together with $f=\int_0^{T^*(f)}p(t)\d t$ to compute
$$\langle f,p^*\rangle=\int_0^{T^*(f)}\langle p(t),p^*\rangle\d t=\int_0^{T^*(f)} J(p^*)\d t=T^*(f)J(p^*),$$
which is implies $T^*(f)={\langle f,p^*\rangle}/{J(p^*)}.$ On the other hand, by Theorem~\ref{thm:equivalence_gf_vp} the solution $u(t)$ of the gradient flow coincides with the solution of the variational problem \eqref{eq:var_prob} which has the extinction time $T^*(f)=\norm{f}_*=\sup_{p\in\calN(J)^\perp}{\langle f,p\rangle}/{J(p)}$ according to \cite{bungert2018solution}. 
\end{proof}
\begin{cor}\label{cor:char_ext_profile}
Under the conditions of Theorem~\ref{thm:formula_ext_time} any extinction profile $p^*$ is a maximizer of the functional 
$$\calN(J)^\perp\ni p\mapsto\frac{\langle f,p\rangle}{J(p)}$$
or, equivalently,
$$\frac{f}{\norm{f}_*}\in\partial J(p^*).$$
In particular, the extinction profile is uniquely determined if $f/\norm{f}_*$ lies in the intersection of $\partial J(0)$ and exactly one other subdifferential.
\end{cor}

\begin{cor}
Under the conditions of Theorem~\ref{thm:formula_ext_time} any extinction profile $p^*$ meets
$$J\left(\frac{p^*}{\norm{p^*}}\right)\leq\frac{\norm{f}}{\norm{f}_*}.$$
This estimate is sharp which can be seen by choosing $f$ to be an eigenvector.
\end{cor}

\begin{example}[One-dimensional total variation flow]
Since according to Section~\ref{sec:1D-TV-flow}, the one-dimensional total variation flow yields a spectral decomposition into eigenvectors, Theorem~\ref{thm:formula_ext_time} and Corollary~\ref{cor:char_ext_profile} are applicable. Bonforte and Figalli in \cite{bonforte2012total} studied the case of a non-negative initial datum $f\in\bv(\R)$ which lives on the (minimally chosen) interval $[a,b]$. They showed that in this case, the extinction time is given by $T=\frac{1}{2}\int_a^b f\d x$ and that $w(t)=u(t)/(T-t)$ converges to $p^*:=\frac{2}{b-a}\chi_{[a,b]}$ for $t\to T$. This follows directly from our results after observing that $f/\norm{f}_*$ can only lie in the subdifferential of a constant function on $[a,b]$ which therefore coincides with $p^*$. The extinction time is consequently given by 
$$T=\frac{\langle f,p^*\rangle}{\tv(p^*)}=\frac{1}{2}\int_a^bf\d x.$$
To see that $g:=f/\norm{f}_*$ cannot lie in the subdifferential of a non-constant function, let us assume that there is $p$ such that $g\in\partial\tv(p)$ and $p'(x_0)\neq 0$ (in the sense of measures) for some $x_0\in(a,b)$. Writing $g=-v'$ with $|v|\leq 1$ and $v=\pm 1$ where $p'$ is positive / negative, we find that $v(x_0)=\pm 1$ and $v$ is non-increasing since $g=-v'$ is non-negative on $\R$. Hence, it holds $v=1$ on $(-\infty,x_0)$ or $v=-1$ on $(x_0,\infty)$ which is a contradiction to the fact that $g=-v'$ takes on strictly positive values on a subset in $[a,x_0]$ and $[x_0,b]$ that has positive Lebesgue measure. Otherwise, the interval $[a,b]$ would not be minimally chosen.
\end{example}
%
%\section*{Open questions}
%\begin{itemize}
%\item Characterization of the extinction profile in terms of $f$ in the general case. It holds $T^*(f)\geq\frac{\langle f,p^*\rangle}{J(p^*)}$
%\item Existence of spectral decompositions 
%\item Minimal element in $\partial K$ eigenvector?
%\item Constructive methods beyond gradient flow
%\end{itemize}

\section*{Acknowledgements}

The work by LB and MB has been supported by ERC via Grant EU FP7 – ERC Consolidator Grant 615216 LifeInverse. MB and AC acknowledge further support by the EU under grant 2020 NoMADS - DLV-777826. MN was partially 
supported by the INDAM/GNAMPA and by the University of Pisa Project PRA 
2017-18.

\bibliographystyle{abbrv}
\bibliography{bibliography}

\end{document}

%% file: eigenvectors.tex
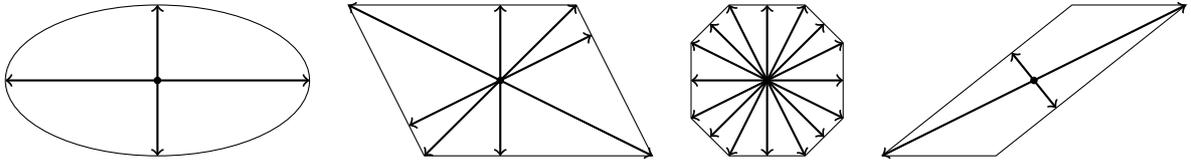
\begin{figure}[h!]
\centering
\begin{tikzpicture}[scale=1.0]
		\node at (0,0) [circle,fill,inner sep=1pt]{};
		\draw (0,0) circle [x radius=2, y radius=1];
		\draw[->,thick](0,0)--(0,1);
        \draw[->,thick](0,0)--(0,-1);
        \draw[->,thick](0,0)--(2,0);
        \draw[->,thick](0,0)--(-2,0);
\end{tikzpicture}
\hfill
\begin{tikzpicture}[scale=1.0]
		\node at (0,0) [circle,fill,inner sep=1pt]{};
        \draw (-1.,-1)--(2,-1)--(1,1)--(-2,1)--(-1.,-1);
        \draw[->,thick](0,0)--(0,1);
        \draw[->,thick](0,0)--(0,-1);
        \draw[->,thick](0,0)--(1,1);
        \draw[->,thick](0,0)--(-1,-1);        
        \draw[->,thick](0,0)--(2,-1);
        \draw[->,thick](0,0)--(-2,1);
        \draw[->,thick](0,0)--(1.2,0.6);
        \draw[->,thick](0,0)--(-1.2,-0.6);       
\end{tikzpicture}
\hfill
\begin{tikzpicture}[scale=1.0]
		\node at (0,0) [circle,fill,inner sep=1pt]{};
		\draw (1,0.5)--(0.5,1)--(-0.5,1)--(-1,0.5)--(-1,-0.5)
		--(-0.5,-1)--(0.5,-1)--(1,-0.5)--(1,0.5);
		\draw[->,thick](0,0)--(0,1);
		\draw[->,thick](0,0)--(0.5,1);
		\draw[->,thick](0,0)--(1,0);
		\draw[->,thick](0,0)--(1,0.5);		
		\draw[->,thick](0,0)--(0.75,0.75);
		\draw[->,thick](0,0)--(1,-0.5);
		\draw[->,thick](0,0)--(0.5,-1);			
		\draw[->,thick](0,0)--(0.75,-0.75);
		
		\draw[->,thick](0,0)--(0,-1);
		\draw[->,thick](0,0)--(-0.5,-1);
		\draw[->,thick](0,0)--(-1,0);
		\draw[->,thick](0,0)--(-1,-0.5);		
		\draw[->,thick](0,0)--(-0.75,-0.75);
		\draw[->,thick](0,0)--(-1,0.5);
		\draw[->,thick](0,0)--(-0.5,1);			
		\draw[->,thick](0,0)--(-0.75,0.75);
\end{tikzpicture}
\hfill
\begin{tikzpicture}[scale=1.0]
		\node at (0,0) [circle,fill,inner sep=1pt]{};
        \draw (-2,-1)--(-0.5,-1)--(2,1)--(0.5,1)--(-2,-1); 
		\draw[->,thick](0,0)--(2,1);	
   		\draw[->,thick](0,0)--(-0.2927,0.3659);      
   		\draw[->,thick](0,0)--(-2,-1);	
   		\draw[->,thick](0,0)--(0.2927,-0.3659);        
\end{tikzpicture}
\caption{\secchange{Four} different dual unit balls $K$ with all eigenvectors with eigenvalue one\label{fig:eigenvectors}}
\end{figure}

%% file: minsub.tex
\begin{figure}[h!]
\centering
\begin{tikzpicture}
		\node at (0,0) [circle,fill,inner sep=1pt]{};	
		\draw (-1.5,-0.5) node[anchor = east] {$\mathcal{S}$};
		\draw[dashed] (-1.5,-0.5)--(0.5,-0.5);	
		\draw[red, thick] (-1,-0.5)--(-0,-0.5);
        \draw[thick] (-0,-0.5)--(1,0.5)--(0,0.5)--(-1,-0.5);
        \draw[->,thick] (0,0)--(-0,-0.5);
        \draw (0,-0.5) node[anchor=north] {${p}$};
\end{tikzpicture}
\hspace*{2cm}
\begin{tikzpicture}
		\node at (0,0) [circle,fill,inner sep=1pt]{};
		\draw (-1.2,-0.15) node[anchor = east] {$\mathcal{S}$};
		\draw[dashed] (-1.2,-0.15)--(0.3,-0.65);		
		\draw[red, thick] (-1.2,-0.5)--(-0.2,-0.5);
        \draw[thick] (-0.2,-0.5)--(1.2,0.5)--(0.2,0.5)--(-1.2,-0.5);
        \draw[->,thick] (0,0)--(-0.2,-0.5);
        \draw (-0.2,-0.5) node[anchor=north] {${p}$};
\end{tikzpicture}
\caption{\textbf{Left:} \eqref{eq:minsub} is met since $\mathcal{S}$ is supporting, \textbf{Right:} here \eqref{eq:minsub} is violated, i.e., $\mathcal{S}$ is not supporting and ${p}$ is no eigenvector\label{fig:minsub}}
\end{figure}
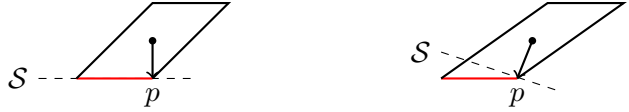